\documentclass[12pt]{article}

\title{
$\VIC$-modules over noncommutative rings}
\author{Andrew Putman\thanks{Supported in part by NSF grant DMS-1811322} \and Steven V Sam\thanks{Supported in part by NSF grant DMS-1849173}}

\date{}

\usepackage{tocloft}
\usepackage{bbm}
\newcommand{\arxiv}[1]{\href{http://arxiv.org/abs/#1}{{\tt arXiv:#1}}}

\usepackage{amsmath,amssymb,amsthm,amscd,amsfonts}
\usepackage{epsfig,pinlabel,paralist}
\usepackage[margin=1in]{geometry}
\usepackage[font=small,format=plain,labelfont=bf,up,textfont=it,up]{caption}
\usepackage{type1cm}
\usepackage{calc}
\usepackage{enumerate}
\usepackage{bm}
\usepackage{microtype}
\usepackage[all,cmtip]{xy}

\usepackage{lmodern}
\usepackage[T1]{fontenc}

\usepackage[bookmarks, bookmarksdepth=2, colorlinks=true, linkcolor=blue, citecolor=blue, urlcolor=blue]{hyperref}

\usepackage{MnSymbol}

\usepackage{etoolbox}
\apptocmd{\thebibliography}{\raggedright}{}{}

\numberwithin{equation}{section}

\newtheorem{stepx}{Step}
\newenvironment{step}[1]
 {\renewcommand\thestepx{#1}\stepx}
 {\endstepx}

\theoremstyle{plain}
\newtheorem{theorem}{Theorem}[section]
\newtheorem{maintheorem}{Theorem}
\newtheorem{proposition}[theorem]{Proposition}
\newtheorem{lemma}[theorem]{Lemma}

\newtheorem*{claim}{Claim}

\theoremstyle{definition}

\newtheorem{defn}[theorem]{Definition}

\theoremstyle{remark}
\newtheorem{rmk}[theorem]{Remark}
\newenvironment{remark}[1][]{\begin{rmk}[#1] \pushQED{\qed}}{\popQED \end{rmk}}

\newtheorem{eg}[theorem]{Example}
\newenvironment{example}[1][]{\begin{eg}[#1] \pushQED{\qed}}{\popQED \end{eg}}

\DeclareMathOperator{\Hom}{Hom}
\DeclareMathOperator{\End}{End}

\DeclareMathOperator{\Mod}{Mod}

\DeclareMathOperator{\GL}{GL}

\DeclareMathOperator{\Mat}{Mat}

\newcommand\R{\ensuremath{\mathbb{R}}}

\newcommand\Z{\ensuremath{\mathbb{Z}}}

\newcommand\Field{\ensuremath{\mathbb{F}}}

\DeclareMathOperator{\HH}{H}

\DeclareMathOperator{\Conf}{Conf}
\DeclareMathOperator{\id}{id}

\DeclareMathOperator{\image}{im}
\newcommand\Set[2]{\ensuremath{\{\text{#1 $|$ #2}\}}}

\newcommand\bbD{\ensuremath{\mathbb{D}}}
\newcommand\bbL{\ensuremath{\mathbb{L}}}

\newcommand\bk{\ensuremath{\mathbf{k}}}

\newcommand\fP{\ensuremath{\mathfrak{P}}}
\newcommand\fS{\ensuremath{\mathfrak{S}}}
\newcommand\fc{\ensuremath{\mathfrak{c}}}
\newcommand\hfc{\ensuremath{\widehat{\fc}}}

\renewcommand\oe{\ensuremath{\overline{e}}}
\newcommand\obbL{\ensuremath{\overline{\bbL}}}

\newcommand\oR{\ensuremath{\overline{R}}}
\newcommand\oS{\ensuremath{\overline{S}}}
\newcommand\ox{\ensuremath{\overline{x}}}
\newcommand\oM{\ensuremath{\overline{M}}}

\newcommand\oh{\ensuremath{\overline{h}}}

\newcommand\og{\ensuremath{\overline{g}}}
\newcommand\oPhi{\ensuremath{\overline{\Phi}}}

\newcommand\hR{\ensuremath{\widehat{R}}}
\newcommand\hr{\ensuremath{\widehat{r}}}

\newcommand\hg{\ensuremath{\widehat{g}}}

\DeclareMathOperator{\FI}{\text{\tt FI}}
\DeclareMathOperator{\tC}{\text{\tt C}}
\DeclareMathOperator{\V}{\text{\tt V}}
\DeclareMathOperator{\VI}{\text{\tt VI}}
\DeclareMathOperator{\VIC}{\text{\tt VIC}}
\DeclareMathOperator{\OVIC}{\text{\tt OVIC}}
\DeclareMathOperator{\init}{init}
\newcommand\sinit{\ensuremath{\mathcal{I}}}

\DeclareMathOperator{\kmod}{\text{$\bk$-Mod}}
\DeclareMathOperator{\groups}{Groups}
\renewcommand{\paragraph}[1]{\bigskip \noindent {\bf #1}}

\begin{document}

\maketitle

\vspace{-24pt}
\begin{abstract}
\noindent
For a finite ring $R$, not necessarily commutative, we prove that the category of $\VIC(R)$-modules over a left Noetherian ring $\bk$ is locally Noetherian,
generalizing a theorem of the authors that dealt with commutative $R$.  As an application, we
prove a very general twisted homology stability for $\GL_n(R)$ with $R$ a finite noncommutative ring.
\end{abstract}


\section{Introduction}

The program of representation stability was introduced
by Church and Farb \cite{ChurchFarbRepStability, FarbICM}.  The idea is that many
of the representations that occur in nature depend on a parameter $n$, and
it is useful to study algebraic structures that encode all of these representations simultaneously.
For instance, the cohomology groups of the space $\Conf_n(\R^2)$ of configurations 
of $n$ labeled points in $\R^2$ are representations
of the symmetric group $S_n$, which acts by permuting the $n$ points.  Individually, these are
hard to understand; however, taken together they have a lot of global structure, especially
as $n \mapsto \infty$.

\paragraph{Representations of categories.}
This can be encoded in many ways.  One of the most fruitful is Church--Ellenberg--Farb's \cite{ChurchEllenbergFarbFI}
theory of $\FI$-modules.  Here $\FI$ is the category whose objects are the finite sets
$[n] = \{1,\ldots,n\}$ and whose morphisms are injections.  For a category $\tC$ like $\FI$
and a ring $\bk$, a {\em $\tC$-module} over $\bk$ is a functor $M$ from
$\tC$ to the category $\kmod$ of left $\bk$-modules.  Thus $M$ consists of a $\bk$-module $M_c$
for every object $c \in \tC$ and a $\bk$-module map $f\colon M_c \rightarrow M_d$ for every $\tC$-morphism
$f\colon c \rightarrow d$.  For an $\FI$-module
$M$, we will write $M_n$ for $M_{[n]}$.  The $\FI$-endomorphisms of $[n]$ form the symmetric group $S_n$.
These endomorphisms act on $M_n$, making each $M_n$ a representation of $S_n$. 

\begin{example}
For a fixed $p$, we can define
an $\FI$-module $M$ over $\Z$ with $M_{n} = \HH^p(\Conf_n(\R^2);\Z)$.  The induced
$S_n = \End_{\FI}([n])$-action on $M_n$ is 
precisely the $S_n$-action on $\HH^p(\Conf_n(\R^2);\Z)$ from the previous
paragraph.  We therefore get a single object encoding all these representations together
with the various ways that they are related as $n \mapsto \infty$.
\end{example}

\paragraph{Homological algebra.}
For a category $\tC$, the collection of $\tC$-modules over $\bk$ forms an abelian category
whose morphisms are natural transformations between functors $\tC \rightarrow \kmod$.  An
important insight of Church--Ellenberg--Farb \cite{ChurchEllenbergFarbFI} is that one can
do commutative and homological algebra in this category in a way that is similar to the category
$\kmod$.  For instance, one can construct projective resolutions, take derived functors, etc.

\paragraph{Local Noetherianity.}
Perhaps the most important technical result for this is a version of the Hilbert
basis theorem.  A $\tC$-module $M$ over a ring $\bk$ is {\em finitely generated} if there exist
objects $c_1,\ldots,c_k \in \tC$ and elements $x_i \in M_{c_i}$ such that the smallest
$\tC$-submodule of $M$ containing all the $x_i$ is $M$.  In other words, for each $c \in \tC$
the set
\[\bigcup_{i=1}^k \Set{$f(x_i)$}{$f\colon c_i \rightarrow c$ is a $\tC$-morphism} \subset M_c.\]
spans $M_c$.  We say that the category of $\tC$-modules over $\bk$ is {\em locally Noetherian}
if for all finitely generated $\tC$-modules $M$ over $\bk$, all
$\tC$-submodules of $M$ are finitely generated.  Generalizing previous work that
dealt for instance with fields $\bk$ of characteristic $0$, Church--Ellenberg--Farb--Nagpal \cite{CEFN}
proved that the category of $\FI$-modules over a left Noetherian ring $\bk$ is locally Noetherian.

\paragraph{VIC-modules.}
The category of $\FI$-modules encodes representations of the symmetric groups, and there has
been a huge amount of work developing analogues for other families of groups (see, e.g., 
\cite{GanLiEI, PutmanSamLinear, PutmanSamSnowdenUni, SamSnowdenGrobner, WilsonFIW}).  One
particularly important family of groups are the general linear groups $\GL_n(R)$ over a ring
$R$.  Here it is natural to look at categories whose objects are the finite-rank free
right\footnote{The purpose of considering $R^n$ as a right $R$-module is that $\GL_n(R)$ acts on $R^n$ on the left
by right $R$-module automorphisms.} $R$-modules $R^n$ with $n \geq 0$.  As for the morphisms, there are several potential choices.
To help keep the notation for our morphisms straight, we will write $[R^n]$ when we
mean to regard $R^n$ as an object of one of our categories and $R^n$ when we mean
to regard it as an $R$-module.
\begin{compactitem}
\item The category $\V(R)$, whose morphisms $[R^n] \rightarrow [R^m]$ are
$R$-linear maps $R^n \rightarrow R^m$.  Versions of this go back to work 
of Lannes and Schwartz and are the focus of the {\em Artinian conjecture} (see \cite[Conjecture 3.12]{Kuhn}),
which was resolved independently by the authors \cite{PutmanSamLinear} and by Sam--Snowden \cite{SamSnowdenGrobner}.
\item The category $\VI(R)$, whose morphisms $[R^n] \rightarrow [R^m]$ are 
injective $R$-linear maps $f\colon R^n \rightarrow R^m$ that are splittable in the sense that
there exists some $g\colon R^m \rightarrow R^n$ with $g \circ f = \text{id}$.  Equivalently, the
image of $f$ is a summand of $R^m$.  This was introduced by Scorichenko in 
his thesis (\cite{ScorichnkoThesis}; see \cite{FraFriPiSch} for a published account).
\item The category $\VIC(R)$, whose morphisms $[R^n] \rightarrow [R^m]$ are pairs $(f_1,f_2)$, where
$f_1\colon R^n \rightarrow R^m$ is an injective $R$-linear map and $f_2\colon R^m \rightarrow R^n$ is
a splitting of $f_1$, so $f_2 \circ f_1 = \text{id}$.  This was introduced by
the authors in \cite{PutmanSamLinear}.
\end{compactitem}

\begin{remark}
One motivation for studying $\VIC(R)$ is that it is the only one of these categories where
there is a functor $\VIC(R) \rightarrow \groups$ taking $R^n \in \VIC(R)$ to $\GL_n(R)$.  For
a morphism $(f_1,f_2)\colon [R^n] \rightarrow [R^m]$, the induced group homomorphism
$\GL_n(R) \rightarrow \GL_m(R)$ is as follows.  Set $C = \ker(f_2)$, so $R^m = \image(f_1) \oplus C$.
Our homomorphism then takes $\phi \in \GL_n(R)$ to the map $R^m \rightarrow R^m$ 
obtained from $f_1 \circ \phi \circ f_1^{-1}\colon \image(f_1) \rightarrow \image(f_1)$ by extending
over $C$ by the identity.
\end{remark}

\begin{remark}
\label{remark:stablefree}
Our definition of $\VIC(R)$ is slightly different from the one in \cite{PutmanSamLinear}, which
requires that a $\VIC(R)$-morphism $(f_1,f_2)$ also have $\ker(f_2)$ free.  For finite (and, more
generally, Artinian) rings,
this added condition is superfluous: $\ker(f_2)$ is in any case stably free, and for Artinian
rings finitely generated stably free modules are free (see \cite[Example I.4.7.3]{LamSerre};
rings with this property are called {\em Hermite rings}).
\end{remark}

\paragraph{Main theorem.}
Fix a left Noetherian ring $\bk$.
In \cite{PutmanSamLinear}, it is proven that for a finite {\em commutative} ring $R$, the categories
of $\V(R)$- and $\VI(R)$- and $\VIC(R)$-modules over $\bk$ are all locally Noetherian (see \cite{SamSnowdenGrobner} for alternate
proofs for $\V(R)$ and $\VI(R)$, but not for $\VIC(R)$).  However, in many situations (e.g., in algebraic K-theory),
it is important to study $\GL_n(R)$ where $R$ is a noncommutative ring.  For instance, $R$ might be the group
ring $\Field_p[G]$ of a finite group $G$.  Our main theorem addresses this more general situation:

\begin{maintheorem}
\label{maintheorem:noetherian}
Let $R$ be a finite ring, not necessarily commutative, and let $\bk$ be
a left Noetherian ring.  Then the categories of $\V(R)$- and
$\VI(R)$- and $\VIC(R)$-modules over $\bk$ are locally Noetherian.
\end{maintheorem}

\begin{remark}
In Theorem \ref{maintheorem:noetherian}, we allow not just the finite rings $R$ to be noncommutative, but
also the base rings $\bk$.  In fact, for $R$ commutative the proof
of Theorem \ref{maintheorem:noetherian} in \cite{PutmanSamLinear} works in that level of generality.
\end{remark}

\begin{remark}
For infinite commutative $R$, the authors proved in \cite{PutmanSamLinear} that the categories
of $\V(R)$- and $\VI(R)$- and $\VIC(R)$-modules over a ring $\bk$ 
are {\em not} locally Noetherian.  The same
argument works for infinite noncommutative $R$.  See \cite{Harman} for one way to get around
this for $R = \Z$.
\end{remark}

\paragraph{Application: twisted homological stability.}
A basic theorem of van der Kallen \cite{vanDerKallen} says that for rings $R$ satisfying
mild hypotheses (for instance, all finite rings), the groups $\GL_n(R)$ satisfy
{\em homological stability}, i.e., for all $p$, we have
\[\HH_p(\GL_n(R);\Z) \cong \HH_p(\GL_{n+1}(R);\Z) \quad \text{for $n \gg p$}.\]
In fact, building on ideas of Dwyer \cite{Dwyer}, van der Kallen is even able to prove
this for certain twisted coefficient systems (those that are ``polynomial'' in an appropriate sense).
For example, he is able to show for all $m \geq 0$ that we have
\[\HH_p(\GL_n(R);(R^n)^{\otimes m}) \cong \HH_p(\GL_{n+1}(R);(R^{n+1})^{\otimes m}) \quad \text{for $n \gg p$}.\]
See \cite{RandalWilliamsWahl} and \cite{PutmanTwistedStability} for alternate proofs of this that use
the language of $\VIC(R)$-modules to encode the twisted coefficients.

In \cite[\S 4]{PutmanSamLinear}, the authors showed how to deduce a much more general version of this
for finite commutative rings from the local Noetherianity of $\VIC(R)$.  Given our new
Theorem \ref{maintheorem:noetherian}, the exact same argument
gives the following result for finite noncommutative rings.  For a $\VIC(R)$-module $M$, write $M_n$
for the value of $M$ on $[R^n] \in \VIC(R)$.  The $\VIC(R)$-endomorphisms of $[R^n]$ form the group $\GL_n(R)$, so
$M_n$ is a representation of $\GL_n(R)$.

\begin{maintheorem}
\label{maintheorem:twisted}
Let $R$ be a finite ring, not necessarily commutative, and let $M$ be a finitely generated $\VIC(R)$-module over a left Noetherian ring $\bk$.
Then for all $p \geq 0$, we have
\[
  \HH_p(\GL_n(R);M_n) \cong \HH_p(\GL_{n+1}(R);M_{n+1})
\]
for $n \gg p$.
\end{maintheorem}

\begin{remark}
The proof of Theorem \ref{maintheorem:twisted} for commutative rings 
in \cite[\S 4]{PutmanSamLinear} uses the more stringent definition of
$\VIC(R)$ discussed in Remark \ref{remark:stablefree}, which as we discussed there
is equivalent to ours for finite rings.
\end{remark}

\begin{remark}
The references \cite{vanDerKallen, RandalWilliamsWahl, PutmanTwistedStability} 
give explicit estimates of when this stability occurs.  Since
we apply our non-effective Noetherianity theorem, we are not able to give such an estimate.  However,
our theorem applies to much more general coefficient systems than the polynomial ones
considered in \cite{vanDerKallen, RandalWilliamsWahl, PutmanTwistedStability}.
\end{remark}

\paragraph{Ideas from proof.}
We will derive Theorem \ref{maintheorem:noetherian} for $\V(R)$ and $\VI(R)$ from the
case of $\VIC(R)$, so we will focus on that category.
In \cite{PutmanSamLinear}, this is dealt with for finite commutative $R$ by a sort of
Gr\"{o}bner basis argument that was introduced to the theory of representation stability
in \cite{SamSnowdenGrobner} (though the general theorems
of \cite{SamSnowdenGrobner} do not apply to $\VIC(R)$; also, we remark that a similar kind of
argument appeared much earlier in work of Richter \cite{Richter}).
We do the same thing, but the details are
far harder.  The main issue is that finite noncommutative rings are much more complicated than
finite commutative rings.  Indeed, the starting point of the proof in \cite{PutmanSamLinear} is the
fact that finite commutative rings are Artinian, and thus are the product of finitely many local rings.
Local rings are not that different from fields, so in the end we can mostly focus on the case of
finite fields.  Unfortunately, noncommutative Artinian rings are not nearly as well-behaved, which 
greatly complicates the proof.

\paragraph{Convention: left vs right modules.}
Throughout this paper, we emphasize that column vectors $R^n$ are considered as right $R$-modules.  With
this convention, the group $\GL_n(R)$ acts on $R^n$ on the left by right $R$-module homomorphisms.  If
we wanted to deal with left $R$-modules, then we would have to use row vectors and have $\GL_n(R)$ act
on the right.

\paragraph{Outline.}
We start in \S \ref{section:reductionordered} by reducing to proving local Noetherianity for an
``ordered'' version of $\VIC(R)$ called $\OVIC(R)$.  The rest of the paper is devoted to this: in
\S \ref{section:artinianrings}, we discuss the structure of finite noncommutative rings, in
\S \ref{section:ovicbasic} we define $\OVIC(R)$ and give its basic properties, and finally
in \S \ref{section:ovicnoetherian} we prove that the category of $\OVIC(R)$-modules is
locally Noetherian.

\begin{remark}
Some parts of our argument are the same as in \cite{PutmanSamLinear}, but we tried to make
this paper mostly self-contained at least for $\VIC(R)$.  The fact that we will focus on
this single category will allow us to write in a much less abstract way, so one side benefit
is that we think some of the details of the proof here will be a little easier to parse.
\end{remark}

\paragraph{Acknowledgments.}
We would like to thank Benson Farb and Andrew Snowden for helpful comments, and 
Peter Patzt for pointing out a small mistake in an earlier version of this paper.

\section{Reduction to ordered \texorpdfstring{$\VIC$}{VIC}}
\label{section:reductionordered}

Instead of working with $\VIC(R)$ directly, our proof will focus on a subcategory $\OVIC(R)$.  The ``O'' stands
for ``ordered''.  Its main properties are as follows:

\begin{theorem}
\label{theorem:orderedsummary}
Let $R$ be a finite ring.  There exists a subcategory $\OVIC(R)$ of $\VIC(R)$ with the following properties:
\begin{compactitem}
\item[(a)] The objects of $\OVIC(R)$ are the same as $\VIC(R)$: the finite-rank free $R$-modules $R^n$ for $n \geq 0$.
\item[(b)] Every $\VIC(R)$-morphism $f\colon [R^d] \rightarrow [R^n]$ can be factored as
\[[R^d] \stackrel{f_1}{\longrightarrow} [R^d] \stackrel{f_2}{\longrightarrow} [R^n],\]
where $f_1\colon [R^d] \rightarrow [R^d]$ is a $\VIC(R)$-morphism and $f_2\colon [R^d] \rightarrow [R^n]$ is an
$\OVIC(R)$-morphism.
\item[(c)] The category of $\OVIC(R)$-modules over a left Noetherian ring $\bk$ is locally Noetherian.
\end{compactitem}
\end{theorem}

\noindent
The proof of Theorem \ref{theorem:orderedsummary} is spread throughout the rest of the paper: in
\S \ref{section:artinianrings}, we discuss some ring-theoretic preliminaries, in \S \ref{section:ovicbasic}
we construct $\OVIC(R)$ and prove part (b) of Theorem \ref{theorem:orderedsummary} (see
Proposition \ref{proposition:factorovic}), and finally in \S \ref{section:ovicnoetherian} we
prove part (c) of Theorem \ref{theorem:orderedsummary} (see Proposition \ref{proposition:ovicnoetherian}).  Here
we will show how to use Theorem \ref{theorem:orderedsummary} to prove Theorem \ref{maintheorem:noetherian}.

\begin{proof}[Proof of Theorem \ref{maintheorem:noetherian}, assuming Theorem \ref{theorem:orderedsummary}]
Let $R$ be a finite ring, not necessarily commutative, and let $\bk$ be a left Noetherian
ring.
In \cite[\S 2.4]{PutmanSamLinear}, the local Noetherianity of the categories of $\V(R)$- and $\VI(R)$-modules over $\bk$
for finite commutative rings $R$
are derived from the local Noetherianity of the category of $\VIC(R)$-modules over $\bk$.  
This derivation does not make use of the commutativity
of $R$, so we must just prove that the category of $\VIC(R)$-modules over $\bk$ 
is locally Noetherian.

Let $M$ be a finitely generated $\VIC(R)$-module over $\bk$.  Our goal is to
prove that every $\VIC(R)$-submodule $N$ of $M$ is finitely generated.  Theorem \ref{theorem:orderedsummary}
says that for the subcategory $\OVIC(R)$ of $\VIC(R)$, the category of $\OVIC(R)$-modules
over $\bk$ is locally Noetherian.  Via restriction, we can regard $M$ as an $\OVIC(R)$-module
and $N$ as an $\OVIC(R)$-submodule of $M$.
We will prove that $M$ is finitely generated as an $\OVIC(R)$-module.
Since the category of $\OVIC(R)$-modules over $\bk$ is locally Noetherian, this will imply
that $N$ is finitely generated as an $\OVIC(R)$-module, and hence a fortiori is also
finitely generated as a $\VIC(R)$-module, as desired.

We will prove that $M$ is finitely generated as an $\OVIC(R)$-module 
by studying representable $\VIC(R)$-modules, which function similarly to free modules.
For $d \geq 0$, let $P(d)$ be the $\VIC(R)$-module defined as follows:
\begin{compactitem}
\item For $n \geq 0$, the $\bk$-module $P(d)_n$ is the free $\bk$-module
with basis $\Hom_{\VIC(R)}(R^d,R^n)$.
\item For a $\VIC(R)$-module morphism $g\colon [R^n] \rightarrow [R^m]$, the
induced $\bk$-module morphism $P(d)_n \rightarrow P(d)_m$ is the one
taking a basis element $f\colon [R^d] \rightarrow [R^n]$ of $P(d)_n$ to
the basis element $g \circ f\colon [R^d] \rightarrow [R^m]$ of $P(d)_m$.
\end{compactitem}
By Theorem \ref{theorem:orderedsummary}, every $\VIC(R)$-module morphism
$f\colon [R^d] \rightarrow [R^n]$ can be factored as
\[[R^d] \stackrel{f_1}{\longrightarrow} [R^d] \stackrel{f_2}{\longrightarrow} [R^n],\]
where $f_1\colon [R^d] \rightarrow [R^d]$ is a $\VIC(R)$-morphism and $f_2\colon [R^d] \rightarrow [R^n]$ is an
$\OVIC(R)$-morphism.  This implies that as an $\OVIC(R)$-module, $P(d)$ is generated by
the set $\Hom_{\VIC(R)}(R^d,R^d) \subset P(d)_d$, which is finite since $R$ is a finite ring.

For all $x \in M_d$ there exists a homomorphism of $\VIC(R)$-modules $P(d) \rightarrow M$ taking the element
$\text{id}\colon [R^d] \rightarrow [R^d]$ of $P(d)_d$ to $x$.  The image of this
$\VIC(R)$-module homomorphism is the $\VIC(R)$-submodule generated by $x$.  
Since $M$ is finitely generated as a $\VIC(R)$-module, for some $d_1,\ldots,d_k \geq 1$ we can find elements $x_i \in M_{d_i}$ 
such that $\{x_1,\ldots,x_k\}$ generates $M$.  Associated to these $x_i$ is a surjective $\VIC(R)$-module homomorphism
\[\bigoplus_{i=1}^k P(d_i) \longrightarrow M.\]
This is, a fortiori, a homomorphism of $\OVIC(R)$-modules, and since each $P(d_i)$ is finitely generated as an $\OVIC(R)$-module, we conclude that $M$ is as well.
\end{proof}

\section{The structure of Artinian rings}
\label{section:artinianrings}

To discuss $\OVIC(R)$, we will need some basic facts about finite rings.  In fact, the
results we need hold more generally for Artinian rings, so we will state them
in this level of generality.  A suitable textbook reference is \cite{LamBook}.  Throughout
this section, $R$ is an Artinian ring.

\paragraph{Peirce decomposition, I.}
We begin with some generalities (see \cite[\S 21]{LamBook}).  Assume that $\{e_1,\ldots,e_{\mu}\}$ are idempotent
elements of $R$ that are orthogonal (i.e., $e_i e_j = 0$ for distinct $1 \leq i,j \leq \mu$) and satisfy
\[1 = e_1 + \cdots + e_{\mu}.\]
Each $e_i R e_j$ is an additive subgroup of $R$, and we have the Peirce decomposition
\begin{equation}
\label{eqn:peirce1}
R = \bigoplus_{i,j=1}^{\mu} e_i R e_j.
\end{equation}
To make this a ring isomorphism, view elements of the right hand side as $\mu \times \mu$ matrices
whose $(i,j)$-entries lie in $e_i R e_j$.  Using the fact that 
\[(e_i R e_k) (e_k R e_j) \subset e_i R e_j \quad \text{and} \quad (e_i R e_k) (e_{k'} R e_j) = 0 \quad \text{if $k \neq k'$},\]
we can multiply these matrices as usual, turning the right hand side of \eqref{eqn:peirce1} into a ring
and \eqref{eqn:peirce1} into a ring isomorphism.  Since $e_i R e_j \subset R$, we can view
\eqref{eqn:peirce1} as an embedding $\Phi\colon R \hookrightarrow \Mat_{\mu}(R)$ that we will
call the {\em Peirce embedding}.\footnote{The fact that the target of the Peirce embedding is $\Mat_{\mu}(R)$ is
just a matter of convenience so we do not have to precisely define a ring of ``matrices'' whose entries all lie in
different places.  Later on we will identify the various $e_i R e_j$ with division rings $\bbD_k$ and even additive
groups $\bbL_{hk}$, and we suggest to the reader that they not focus too much on how these lie inside $R$.}

\paragraph{Peirce decomposition, II.}
Continue with the notation of the previous paragraph.  A more conceptual way to think about
the Peirce embedding is as follows.  Each $e_i R$ is a right $R$-module, and letting 
$R_R$ denote $R$ considered as a right $R$-module we have
\[R_R = \bigoplus_{i=1}^{\mu} e_i R.\]
Via left multiplication, the ring $R$ acts on the left\footnote{This is not a typo -- we are using
the fact that the left and right actions of $R$ on itself commute, i.e., that $R$ is an $(R,R)$-bimodule.} on $R_R$ by right $R$-module endomorphisms, and in fact
$R \cong \End(R_R)$.  We thus have
\begin{equation}
\label{eqn:peirce2}
R = \End(R_R) = \bigoplus_{i,j=1}^{\mu} \Hom(e_j R, e_i R).
\end{equation}
For all $1 \leq i,j \leq \mu$, we have $\Hom(e_j R, e_i R) = e_i R e_j$, where $\phi \in \Hom(e_j R,e_i R)$ corresponds
to the element\footnote{Note that $\phi(e_j) = \phi(e_j^2) = \phi(e_j) e_j = e_i r e_j$ for some $r \in R$.} $\phi(e_j) \in e_i R e_j$.  Making these identifications turns \eqref{eqn:peirce2} into
\eqref{eqn:peirce1}.  This makes it clear that the Peirce embedding reflects the left action of $R$ on $R_R$; indeed,
using
\[R_R = \bigoplus_{i=1}^{\mu} e_i R,\]
we can embed $R_R$ into the set of length-$\mu$ column vectors $R^{\mu}$, which is itself a right $R$-module.
The matrices $\Mat_{\mu}(R)$ act on the left on $R^{\mu}$ by right $R$-module endomorphisms, and we have a commutative diagram
\[\begin{CD}
R    @>{\cong}>> \End(R_R) \\
@V{\Phi}VV             @VVV \\
\Mat_{\mu}(R) @>{\cong}>> \End(R^{\mu}).
\end{CD}\]

\paragraph{Jacobson radical.}
Let $J(R)$ be the Jacobson radical of $R$, i.e., the intersection of all left ideals of $R$.  Alternatively,
$J(R)$ consists of all $y \in R$ such that for all $x,z \in R$, the element $1-xyz$ is a unit  (see \cite[Lemma 4.3]{LamBook}).  
Since $R$ is Artinian,
$J(R)$ can also be characterized as the largest ideal of $R$ that is nilpotent, i.e., such that
$J(R)^k = 0$ for $k \gg 0$ (see \cite[Theorem 4.12]{LamBook}).  Let $\oR = R/J(R)$.  For
$x \in R$, let $\ox \in \oR$ be its image.  Also, for a matrix $M \in \Mat_{n,m}(R)$, let
$\oM \in \Mat_{n,m}(\oR)$ be its image.  The following simple fact will be important for
us.

\begin{lemma}
\label{lemma:checkinvertible}
Let $R$ be a ring and let $M \in \Mat_{n}(R)$ for some $n \geq 1$.  Then $M$ is invertible if
and only if $\oM$ is invertible.
\end{lemma}
\begin{proof}
We have $J(\Mat_n(R)) = \Mat_n(J(R))$ (see \cite[p.\ 61]{LamBook}), so $\overline{\Mat_n(R)} = \Mat_n(\oR)$.  
The result now follows from the fact that for {\em any} ring $S$ (including, in particular, $S = \Mat_n(R)$), 
an element $x \in S$ is invertible if and only if $\ox \in \oS$ is invertible (see \cite[Proposition 4.8]{LamBook}).
\end{proof}

\paragraph{Artin--Wedderburn.} 
The fact that $R$ is Artinian implies that $\oR$ is semisimple (see \cite[Theorem 4.14]{LamBook}),
which by the Artin--Wedderburn Theorem \cite[Theorem 3.5]{LamBook} means that
\begin{equation}
\label{eqn:artinwedderburn}
\oR \cong \Mat_{\mu_1}(\bbD_1) \times \cdots \times \Mat_{\mu_{q}}(\bbD_q)
\end{equation}
for division rings $\bbD_1,\ldots,\bbD_{q}$.
We remark that when $R$ is finite as it is in
most of this paper, Wedderburn's Little Theorem \cite[Theorem 13.1]{LamBook} implies that
the $\bbD_k$ are actually (commutative) fields.
The decomposition \eqref{eqn:artinwedderburn} arises from orthogonal idempotents $\oe_i^{k} \in \oR$ for
$1 \leq k \leq q$ and $1 \leq i \leq \mu_k$ satisfying
\begin{equation}
\label{eqn:artinwedderburnidem}
1 = (\oe^1_1+\cdots+\oe^1_{\mu_1})+\cdots+(\oe^q_1+\cdots+\oe^q_{\mu_q})
\end{equation}
as well as the following:
\begin{compactitem}
\item For $1 \leq k \leq q$ and $1 \leq i,j \leq \mu_k$, we have $\oe_i^k \oR \oe_j^k \cong \bbD_k$.  This is an isomorphism
of rings if $i = j$ (so $\oe_i^k \oR \oe_j^k$ is closed under multiplication), and an isomorphism of additive groups
if $i \neq j$.
\item For $1 \leq k,k' \leq q$ and $1 \leq i,j \leq \mu_k$ with $k \neq k'$, we have $\oe_i^k \oR \oe_j^{k'} = 0$.
\end{compactitem}
Setting $\mu = \mu_1+\cdots+\mu_q$, the Peirce embedding associated to \eqref{eqn:artinwedderburnidem} 
is precisely the embedding $\oPhi\colon \oR \hookrightarrow \Mat_{\mu}(\oR)$ taking an element of
$\oR$ to the matrices in \eqref{eqn:artinwedderburn}, arranged as diagonal blocks in $\Mat_{\mu}(\oR)$.  Here
we are identifying the $\bbD_k$ with appropriate subrings and additive subgroups of $\oR$ as above.

\paragraph{Lifting idempotents.}
Since $J(R)$ is nilpotent, idempotents in $\oR$ can be lifted to $R$ (see \cite[Theorem 21.28]{LamBook}; we remark
that a ring $R$ such that $\oR$ is semisimple and all idempotents in $\oR$ can be lifted to $R$ is called
{\em semiperfect}).
Combined with \cite[Proposition 21.25]{LamBook} and the proof of \cite[Theorem 23.6]{LamBook}, this implies
we can find orthogonal idempotents $e_i^k \in R$ for $1 \leq k \leq q$ and $1 \leq i \leq \mu_k$ lifting
the $\oe_i^k$ such that
\begin{equation}
\label{eqn:bigpeirce}
1= (e^1_1+\cdots+e^1_{\mu_1})+\cdots+(e^q_1+\cdots+e^q_{\mu_q}).
\end{equation}
What is more, by \cite[Proposition 21.21]{LamBook}, we have
\begin{equation}
\label{eqn:ourisomorphisms}
e_i^k R \cong e_{i'}^{k'} R \quad \Leftrightarrow \quad \oe_i^k \oR \cong \oe_{i'}^{k'} \oR \quad \Leftrightarrow \quad k=k'.
\end{equation}
For $1 \leq h,k \leq q$, pick some $1 \leq i \leq \mu_h$ and $1 \leq j \leq \mu_k$ and set
\[\bbL_{hk} = e_i^{h} R e_j^{k} \cong \Hom(e_i^k R, e_j^{h} R).\]
This is a ring if $i=j$ and $k = h$, and an additive group otherwise.
By \eqref{eqn:ourisomorphisms}, up to isomorphisms of rings or additive groups this does not depend on the choice of $i$ and $j$.
In particular, by taking $i = j$ we can identify $\bbL_{kk}$ with a ring.  Observe the following:
\begin{compactitem}
\item For distinct $h$ and $k$ the additive group $\bbL_{hk}$ projects to $0$ in $\oR$, so $\bbL_{hk}$ is an additive subgroup
of the Jacobson radical $J(R)$.
\item As in the previous paragraph, we can identify $\bbD_k$ with a corresponding subring of $\oR$ such that $\bbL_{kk}$ projects to
$\bbD_k$ in $\oR$.  By \cite[Theorem 19.1]{LamBook}, this implies that the rings $\bbL_{kk}$ are local rings, i.e., have a
unique maximal left ideal (or equivalently, a unique maximal right ideal; see \cite[Theorem 19.1]{LamBook}).
\end{compactitem}

\paragraph{Summary.}
Recall that $\mu = \mu_1+\cdots+\mu_q$.  Using the isomorphisms \eqref{eqn:ourisomorphisms}, the 
Peirce embedding $\Phi\colon R \hookrightarrow \Mat_{\mu}(R)$
associated to \eqref{eqn:bigpeirce} can be identified with a ring homomorphism that
takes $x \in R$ to a $q \times q$ block matrix of the form
\[\Phi(x) = \left(\Phi_{hk}(x)\right)_{h,k=1}^q \quad \text{with} \quad \Phi_{hk}(x) \in \Mat_{\mu_h,\mu_k}(\bbL_{hk}).\]
Here we are identifying the $\bbL_{hk}$ with appropriate subrings and additive subgroups of $R$.  Moreover, identifying
the $\bbD_k$ with appropriate subrings and additive subgroups of $\oR$, we have
\[\overline{\Phi(x)} = \overline{\Phi}(\overline{x}) = \left(\overline{\Phi_{hk}(x)}\right)_{h,k=1}^q \quad \text{with} \quad
\overline{\Phi_{hh}(x)} \in \Mat_{\mu_h}(\bbD_h) \text{ and } \overline{\Phi_{hk}(x)}=0 \text{ for $h \neq k$}.\]
We will call this the {\em Artin--Wedderburn embedding} of $R$.

\section{Ordered \texorpdfstring{$\VIC$}{VIC}: definition and basic properties}
\label{section:ovicbasic}

This section defines the subcategory $\OVIC(R)$ of $\VIC(R)$ and proves some basic facts about it.
We do this in two steps: in \S \ref{section:ovicsemisimple}, we deal with semisimple rings, and in
\S \ref{section:ovicgeneral} we deal with Artinian rings (and thus general finite rings).

\subsection{Ordered \texorpdfstring{$\VIC$}{VIC} for semisimple rings}
\label{section:ovicsemisimple}

We start by introducing the notation we will use in this section.
Let $R$ be a semisimple ring, so
\begin{equation}
\label{eqn:decompsemisimple}
R \cong \Mat_{\mu_1}(\bbD_1) \times \cdots \times \Mat_{\mu_q}(\bbD_q)
\end{equation}
for $\mu_1,\ldots,\mu_q \geq 1$ and division rings $\bbD_1,\ldots,\bbD_q$.
Set $\mu = \mu_1+\cdots+\mu_q$ and let $\Phi\colon R \hookrightarrow \Mat_{\mu}(R)$ be the Artin--Wedderburn
embedding of $R$.  For $x \in R$, the matrix $\Phi(x)$ thus consists of the matrices in
\eqref{eqn:decompsemisimple}, arranged as diagonal blocks in $\Mat_{\mu}(R)$.

\paragraph{An example.}
What we will do requires a lot of notation, so to help the reader follow it we start with an example.  Assume
that the decomposition \eqref{eqn:decompsemisimple} is of the form
\[R \cong \Mat_2(\bbD_1) \times \Mat_3(\bbD_2),\]
so the Artin--Wedderburn embedding of $R$ is of the form $\Phi\colon R \hookrightarrow \Mat_{5}(R)$.  Consider
an $R$-linear map $h\colon R^3 \rightarrow R^2$.  The map $h$ can be represented by a $2 \times 3$ matrix
with entries in $R$.  Applying $\Phi$ to the entries of this matrix results in a block matrix $\Phi(h)$ of the form
\[\Phi(h) = \left(\begin{array}{@{}ccccc|ccccc|ccccc@{}}
\thinstar_{11} & \thinstar_{12} &                 &                 &                 & \thinstar_{13} & \thinstar_{14} &            &            &                           & \thinstar_{15} & \thinstar_{16} &                 &                 &            \\
\thinstar_{21} & \thinstar_{22} &                 &                 &                 & \thinstar_{23} & \thinstar_{24} &            &            &                           & \thinstar_{25} & \thinstar_{26} &                 &                 &            \\
          &                     & \smallstar_{11} & \smallstar_{12} & \smallstar_{13} &                &                & \smallstar_{14} & \smallstar_{15} & \smallstar_{16} &                &                & \smallstar_{17} & \smallstar_{18} & \smallstar_{19} \\
          &                     & \smallstar_{21} & \smallstar_{22} & \smallstar_{23} &                &                & \smallstar_{24} & \smallstar_{25} & \smallstar_{26} &                &                & \smallstar_{27} & \smallstar_{28} & \smallstar_{29} \\
          &                     & \smallstar_{31} & \smallstar_{32} & \smallstar_{33} &                &                & \smallstar_{34} & \smallstar_{35} & \smallstar_{36} &                &                & \smallstar_{37} & \smallstar_{38} & \smallstar_{39} \\
\hline
\thinstar_{31} & \thinstar_{32} &                 &                 &                 & \thinstar_{33} & \thinstar_{34} &            &            &                           & \thinstar_{35} & \thinstar_{36} &                 &                 &            \\
\thinstar_{41} & \thinstar_{42} &                 &                 &                 & \thinstar_{43} & \thinstar_{44} &            &            &                           & \thinstar_{45} & \thinstar_{46} &                 &                 &            \\
          &                     & \smallstar_{41} & \smallstar_{42} & \smallstar_{43} &                &                & \smallstar_{44} & \smallstar_{45} & \smallstar_{46} &                &                & \smallstar_{47} & \smallstar_{48} & \smallstar_{49} \\
          &                     & \smallstar_{51} & \smallstar_{52} & \smallstar_{53} &                &                & \smallstar_{54} & \smallstar_{55} & \smallstar_{56} &                &                & \smallstar_{57} & \smallstar_{58} & \smallstar_{59} \\
          &                     & \smallstar_{61} & \smallstar_{62} & \smallstar_{63} &                &                & \smallstar_{64} & \smallstar_{65} & \smallstar_{66} &                &                & \smallstar_{67} & \smallstar_{68} & \smallstar_{69}
\end{array}\right),\]
where the $\thinstar_{ij}$ are elements of $\bbD_1$, the $\smallstar_{ij}$ are elements of $\bbD_2$, and the
blank entries are zeros.  Here we are identifying
the $\bbD_k$ with appropriate subrings and additive subgroups of $R$ as in the previous section.  Let $h_1$ and $h_2$ be the submatrices of $\Phi(h)$ consisting of entries lying in $\bbD_1$ and $\bbD_2$, respectively, so
\[h_1 = \left(\begin{array}{@{}cc|cc|cc@{}}
\thinstar_{11} & \thinstar_{12} & \thinstar_{13} & \thinstar_{14} & \thinstar_{15} & \thinstar_{16} \\
\thinstar_{21} & \thinstar_{22} & \thinstar_{23} & \thinstar_{24} & \thinstar_{25} & \thinstar_{26} \\
\hline
\thinstar_{31} & \thinstar_{32} & \thinstar_{33} & \thinstar_{34} & \thinstar_{35} & \thinstar_{36} \\
\thinstar_{41} & \thinstar_{42} & \thinstar_{43} & \thinstar_{44} & \thinstar_{45} & \thinstar_{46}
\end{array}\right)
\quad
h_2 = \left(\begin{array}{@{}ccc|ccc|ccc@{}}
\smallstar_{11} & \smallstar_{12} & \smallstar_{13} & \smallstar_{14} & \smallstar_{15} & \smallstar_{16} & \smallstar_{17} & \smallstar_{18} & \smallstar_{19} \\
\smallstar_{21} & \smallstar_{22} & \smallstar_{23} & \smallstar_{24} & \smallstar_{25} & \smallstar_{26} & \smallstar_{27} & \smallstar_{28} & \smallstar_{29} \\
\smallstar_{31} & \smallstar_{32} & \smallstar_{33} & \smallstar_{34} & \smallstar_{35} & \smallstar_{36} & \smallstar_{37} & \smallstar_{38} & \smallstar_{39} \\
\hline
\smallstar_{41} & \smallstar_{42} & \smallstar_{43} & \smallstar_{44} & \smallstar_{45} & \smallstar_{46} & \smallstar_{47} & \smallstar_{48} & \smallstar_{49} \\
\smallstar_{51} & \smallstar_{52} & \smallstar_{53} & \smallstar_{54} & \smallstar_{55} & \smallstar_{56} & \smallstar_{57} & \smallstar_{58} & \smallstar_{59} \\
\smallstar_{61} & \smallstar_{62} & \smallstar_{63} & \smallstar_{64} & \smallstar_{65} & \smallstar_{66} & \smallstar_{67} & \smallstar_{68} & \smallstar_{69}
\end{array}\right).\] 
Regard $\Phi(h)$ as an $R$-linear map $\Phi(h)\colon R^{15} \rightarrow R^{10}$.  Let
\[V_1 = \{\vec{v}(1)_1,\ldots,\vec{v}(1)_6\} \quad \text{and} \quad
V_2 = \{\vec{v}(2)_1,\ldots,\vec{v}(2)_9\}\]
be the standard basis elements of $R^{15}$ corresponding to the columns of $\Phi(h)$ where the entries must lie
in $\bbD_1$ and $\bbD_2$, respectively, ordered using the natural ordering on the standard basis elements of $R^{15}$.  In its
natural ordering, the standard basis for $R^{15}$ is thus
\begin{footnotesize}
\[\{\vec{v}(1)_1,\vec{v}(1)_2,\vec{v}(2)_1,\vec{v}(2)_2,\vec{v}(2)_3,\vec{v}(1)_3,\vec{v}(1)_4,\vec{v}(2)_4,\vec{v}(2)_5,\vec{v}(2)_6,\vec{v}(1)_5,\vec{v}(1)_6,\vec{v}(2)_7,\vec{v}(2)_8,\vec{v}(2)_9\}.\]
\end{footnotesize}
  Similarly, let
\[W_1 = \{\vec{w}(1)_1,\ldots,\vec{w}(1)_4\} \quad \text{and} \quad
W_2 = \{\vec{w}(2)_1,\ldots,\vec{w}(2)_6\}\]
be the standard basis elements of $R^{10}$ corresponding to the rows of $\Phi(h)$ where the entries must lie in $\bbD_1$ and $\bbD_2$, respectively.
For $k=1,2$, the linear map $\Phi(h)$ takes each element of $V_k$ to a linear combination of elements of $W_k$, and the resulting
linear map from the span of the $V_k$ to the span of the $W_k$ is described by the matrix $h_k$.  We will call
$V_1 \cup V_2$ and $W_1 \cup W_2$ the {\em distinguished bases} of $R^{15}$ and $R^{10}$, respectively.

\paragraph{Decomposing maps and labeling rows/columns.}
We now consider the general case, so $R$ is of the form
\begin{equation}
\label{eqn:decompsemisimple2}
R \cong \Mat_{\mu_1}(\bbD_1) \times \cdots \times \Mat_{\mu_q}(\bbD_q)
\end{equation}
for $\mu_1,\ldots,\mu_q \geq 1$ and division rings $\bbD_1,\ldots,\bbD_q$.
Set $\mu = \mu_1+\cdots+\mu_q$ and let $\Phi\colon R \hookrightarrow \Mat_{\mu}(R)$ be the Artin--Wedderburn
embedding of $R$. Consider an $R$-linear map $h\colon R^m \rightarrow R^n$.  Regard $h$ as an $n \times m$ matrix
with entries in $R$, and let $\Phi(h)\colon R^{\mu m} \rightarrow R^{\mu n}$ be the linear map obtained by
applying $\Phi$ to each entry in this matrix.  Just as above, for $1 \leq k \leq q$ we can extract
submatrices $h_k\colon R^{\mu_k m} \rightarrow R^{\mu_k n}$ of $\Phi(h)$ consisting of the entries
of $\Phi(h)$ that are required to lie in $\bbD_k$, which is identified with appropriate subrings and additive
subgroups of $R$.  

The {\em labeling} on the rows of $\Phi(h)$ is the association of the pair $(k,i)$ with
$1 \leq k \leq q$ and $i \in \{1,\ldots,\mu_k n\}$ to the row of $\Phi(h)$ corresponding to the $i^{\text{th}}$
row of $h_k$.  Similarly, the {\em labeling} on the columns of $\Phi(h)$ is the association of the pair $(k,j)$
with $1 \leq k \leq q$ and $j \in \{1,\ldots,\mu_k m\}$ to the column of $\Phi(h)$ corresponding to the $j^{\text{th}}$
row of $h_k$.

\paragraph{Distinguished bases.}
We will need notation for the collections of basis elements of $R^{\mu m}$ and $R^{\mu n}$ corresponding
to these submatrices.  The {\em distinguished basis of $R^{\mu m}$} is defined as follows.  For each $1 \leq k \leq q$, let
$\{\vec{v}(k)_{1},\ldots,\vec{v}(k)_{\mu_k m}\}$ be the portion of the standard basis of $R^{\mu m}$
corresponding to the columns of $\Phi(h)$ that are labeled by $(k,j)$ for some $j \in \{1,\ldots,\mu_k m\}$,
arranged in their natural increasing order.  In its natural ordering, the standard basis for $R^{\mu m}$ is thus
\[\vec{v}(1)_1,\ldots,\vec{v}(1)_{\mu_1},\vec{v}(2)_1,\ldots,\vec{v}(2)_{\mu_2},\ldots,\vec{v}(q)_1,\ldots,\vec{v}(q)_{\mu_q}\]
followed by
\[\vec{v}(1)_{\mu_1+1},\ldots,\vec{v}(1)_{\mu_1+\mu_1},\vec{v}(2)_{\mu_2+1},\ldots,\vec{v}(2)_{\mu_2+\mu_2},\ldots,\vec{v}(q)_{\mu_q+1},\ldots,\vec{v}(q)_{\mu_q+\mu_q},\]
etc.,\ finally ending with
\[\vec{v}(1)_{(m-1)\mu_1+1},\ldots,\vec{v}(1)_{(m-1)\mu_1+\mu_1},\ldots,\vec{v}(q)_{(m-1)\mu_q+1},\ldots,\vec{v}(q)_{(m-1)\mu_q+\mu_q}.\]
Similarly, the {\em distinguished basis of $R^{\mu n}$} is defined by letting
$\{\vec{w}(k)_1,\ldots,\vec{w}(k)_{\mu_k n}\}$ for $1 \leq k \leq q$ be the portion of the standard basis
of $R^{\mu n}$ corresponding to the rows of $\Phi(h)$ labeled by $(k,i)$ for some $i \in \{1,\ldots,\mu_k n\}$,
arranged in their natural increasing order. 
For all $1 \leq k \leq q$ and $0 \leq j \leq \mu_k m$, we thus have
\[\Phi(h)(\vec{v}(k)_j) \subset \bigoplus_{i=1}^{\mu_k n} \vec{w}(k)_i \cdot \bbD_k.\]

\paragraph{Surjective maps.}
Now assume that $h\colon R^m \rightarrow R^n$ is a surjective $R$-linear map.  The maps 
$h_k\colon \bbD_k^{\mu_k m} \rightarrow \bbD_k^{\mu_k n}$ discussed above are thus also surjective.
Recall that linear algebra over division rings is similar to linear algebra over fields.  In particular,
notions of basis, dimension, etc.\ make sense in this noncommutative context.  
Considerations of dimension show that there exists some subset $S \subset \{1,\ldots,\mu_k m\}$ such that
$\Set{$\Phi(h)(\vec{v}(k)_i)$}{$i \in S$}$ is a basis for the $\bbD_k$-submodule of
$R^{\mu_k n}$ spanned by $\{\vec{w}(k)_1,\ldots,\vec{w}(k)_{\mu_k n}\}$.  Define
$\fS(h,k)$ to be the smallest such $S \subset \{1,\ldots,\mu_k m\}$ in the lexicographic ordering.
In this ordering, for distinct
$S_1,S_2 \subset \{1,\ldots,\mu_k m\}$ we have $S_1 < S_2$ if the minimal element in the symmetric
difference of the $S_i$ lies in $S_1$.  
The following lemma gives an alternate characterization
of $\fS(h,k)$:

\begin{lemma}
\label{lemma:identifys}
Let $h\colon R^m \rightarrow R^n$ be a surjective $R$-linear map.  For $1 \leq k \leq q$, write
$\fS(h,k) = \{j_1 < j_2 < \cdots < j_{\mu_k n}\}$.  Then the $j_i$ are the unique elements of
$\{1,\ldots,\mu_k m\}$ satisfying the following two conditions:
\begin{compactitem}
\item $\{\Phi(h)(\vec{v}(k)_{j_1}),\ldots,\Phi(h)(\vec{v}(k)_{j_{\mu_k n}})\}$
is a basis for the $\bbD_k$-module
\[\bigoplus_{i=1}^{\mu_k n} \vec{w}(k)_i \cdot \bbD_k.\]
\item Consider $1 \leq j \leq \mu_k m$, and let $1 \leq i_0 \leq \mu_k n$
be the largest index such that $j_{i_0} \leq j$.  Then
\[\Phi(h)(\vec{v}(k)_j) \in \bigoplus_{i=1}^{i_0} \Phi(h)(\vec{v}(k)_{j_i}) \cdot \bbD_k.\]
\end{compactitem}
\end{lemma}
\begin{proof}
Immediate.
\end{proof}

\paragraph{Column-adapted maps.}
This allows us to make the following definition.  A surjective $R$-linear map $h\colon R^m \rightarrow R^n$ 
is {\em column-adapted} if it satisfies the following condition for each $1 \leq k \leq q$.
Write $\fS(h,k) = \{j_1 < j_2 < \cdots < j_{\mu_k n}\}$.  We then require that
$\Phi(h)(\vec{v}(k)_{j_i}) = \vec{w}(k)_{i}$ for all $1 \leq i \leq \mu_k n$.  In other words,
the matrix corresponding to the linear map $h_k\colon \bbD_k^{\mu_k m} \rightarrow \bbD_k^{\mu_k n}$ discussed
above is in reduced row echelon form.  

\begin{example}
Let us return to the example in the beginning of this section where
\[R \cong \Mat_2(\bbD_1) \times \Mat_3(\bbD_2).\]
An $h\colon R^3 \rightarrow R^2$ satisfying the following is column-adapted:
\[\Phi(h) = \left(\begin{array}{@{}ccccc|ccccc|ccccc@{}}
1 & \thinstar &   &   &            & 0 & \thinstar &   &   &            & 0 & 0 &   &   &   \\
0 &         0 &   &   &            & 1 & \thinstar &   &   &            & 0 & 0 &   &   &   \\
  &           & 0 & 1 & \smallstar &   &           & 0 & 0 & \smallstar &   &   & 0 & 0 & 0 \\
  &           & 0 & 0 &          0 &   &           & 1 & 0 & \smallstar &   &   & 0 & 0 & 0 \\
  &           & 0 & 0 &          0 &   &           & 0 & 1 & \smallstar &   &   & 0 & 0 & 0 \\
\hline
0 &         0 &   &   &            & 0 &         0 &   &   &            & 1 & 0 &  &    &   \\
0 &         0 &   &   &            & 0 &         0 &   &   &            & 0 & 1 &  &    &   \\
  &           & 0 & 0 &          0 &   &           & 0 & 0 &          0 &   &   & 1 & 0 & 0 \\
  &           & 0 & 0 &          0 &   &           & 0 & 0 &          0 &   &   & 0 & 1 & 0 \\
  &           & 0 & 0 &          0 &   &           & 0 & 0 &          0 &   &   & 0 & 0 & 1
\end{array}\right).\]
Here the $\thinstar$ are elements of $\bbD_1$, the $\smallstar$ are elements of $\bbD_2$, and the
blank entries are zeros.  The $\fS(h,k)$ are $\fS(h,1) = \{1,3,5,6\}$ and $\fS(h,2) = \{2,4,5,7,8,9\}$.
\end{example}

As is well-known, the set
of reduced row echelon matrices is closed under multiplication, and similarly the class
of column-adapted maps is closed under composition:

\begin{lemma}
\label{lemma:composecoladapted}
Let $g \colon R^m \rightarrow R^n$ and $h \colon R^n \rightarrow R^{\ell}$ be column-adapted
maps.  Then $h \circ g \colon R^m \rightarrow R^{\ell}$ is column-adapted.
\end{lemma}
\begin{proof}
Let $\vec{v}(k)_i$ and $\vec{w}(k)_i$ and $\vec{u}(k)_i$ be the distinguished bases for
$R^{\mu m}$ and $R^{\mu n}$ and $R^{\mu \ell}$, respectively.  Fix some $1 \leq k \leq q$, and write
\begin{align*}
\fS(g,k) &= \{j_1 < j_2 < \cdots < j_{\mu_k n}\}, \\
\fS(h,k) &= \{j'_1 < j'_2 < \cdots < j'_{\mu_k \ell}\}.
\end{align*}
For $1 \leq i \leq \mu_k \ell$, define $j''_i = j_{j'_i}$.  We thus have
\begin{equation}
\label{eqn:jprimeprime}
\{j''_1 < j''_2 < \cdots < j''_{\mu_k \ell}\}
\end{equation}
and
\[h \circ g(\vec{v}(k)_{j''_i}) = h \circ g(\vec{v}(k)_{j_{j'_i}}) = h(\vec{w}(k)_{j'_i}) = \vec{u}(k)_i.\]
From this, it is easy to see that \eqref{eqn:jprimeprime} satisfies the criterion
of Lemma \ref{lemma:identifys}, so $\fS(h \circ g,k)$ equals \eqref{eqn:jprimeprime} and
$h \circ g$ is column-adapted.
\end{proof}

For later use, we record a corollary of the proof of this lemma.

\begin{lemma}
\label{lemma:ordersfunction}
Let $m \geq n \geq \ell$ and $1 \leq k \leq q$.  Order subsets of $\{1,\ldots,\mu_k m\}$ and
$\{1,\ldots,\mu_k n\}$ lexicographically.
Let $f\colon R^m \rightarrow R^n$ and $g,h\colon R^n \rightarrow R^{\ell}$ be column-adapted maps
such that $\fS(g,k) \leq \fS(h,k)$.  Then $\fS(g \circ f,k) \leq \fS(h \circ f,k)$, with equality if
and only if $\fS(g,k) = \fS(h,k)$.
\end{lemma}
\begin{proof}
Immediate from the proof of Lemma \ref{lemma:composecoladapted}.
\end{proof}

\paragraph{Ordered VIC, semisimple case.}
By Lemma \ref{lemma:composecoladapted}, it makes sense to define $\OVIC(R)$ to be the subcategory of $\VIC(R)$ whose
objects are all the $R^n$ with $n \geq 1$ and whose morphisms $f\colon [R^n] \rightarrow [R^m]$
are all the $\VIC(R)$-morphisms $f = (f',f'')$ such that $f''$ is column-adapted.  Since
the only column-adapted maps $R^n \rightarrow R^n$ are the identity, it follows that
the identity is the only $\OVIC(R)$-endomorphism of $[R^n]$.  In the next section, we will
show how to generalize all of this to the case of Artinian $R$, and thus in particular to
all finite $R$.

\subsection{Ordered \texorpdfstring{$\VIC$}{VIC} for general Artinian rings}
\label{section:ovicgeneral}

Let $R$ be an Artinian ring.  The structure of $R$ was discussed in \S \ref{section:artinianrings}, and we
briefly recall it.  The quotient ring $\oR = R/J(R)$ is semisimple, so
\[\oR \cong \Mat_{\mu_1}(\bbD_1) \times \cdots \times \Mat_{\mu_q}(\bbD_q)\]
for $\mu_1,\ldots,\mu_q \geq 1$ and division rings $\bbD_1,\ldots,\bbD_q$.
Set $\mu = \mu_1+\cdots+\mu_q$.  Let $\Phi\colon R \hookrightarrow \Mat_{\mu}(R)$ and
$\oPhi\colon \oR \hookrightarrow \Mat_{\mu}(\oR)$ be the Artin--Wedderburn
embeddings of $R$ and $\oR$, so $\overline{\Phi(x)} = \oPhi(\ox)$ for all $x \in R$.  Also,
for $1 \leq h,k \leq q$ let $\bbL_{hk}$ be as defined in \S \ref{section:artinianrings},
so the $\bbL_{kk}$ are local rings and the $\obbL_{hk}$ for distinct $h$ and $k$ are additive subgroups of the Jacobson radical $J(R)$.
The Artin--Wedderburn embedding $\Phi\colon R \hookrightarrow \Mat_{\mu}(R)$ can then be decomposed
into a $q \times q$ block matrix of the form
\[\Phi(x) = \left(\Phi_{hk}(x)\right)_{h,k=1}^q \quad \text{with} \quad \Phi_{hk}(x) \in \Mat_{\mu_h,\mu_k}(\bbL_{hk}),\]
and
\[\overline{\Phi(x)} = \overline{\Phi}(\overline{x}) = \left(\overline{\Phi_{hk}(x)}\right)_{h,k=1}^q \quad \text{with} \quad
\overline{\Phi_{hh}(x)} \in \Mat_{\mu_h}(\bbD_h) \text{ and } \overline{\Phi_{hk}(x)}=0 \text{ for $h \neq k$}.\]
Here the $\bbL_{hk}$ are embedded as appropriate subrings and additive subgroups of $R$, and similarly the
$\bbD_k$ are embedded as appropriate subrings and additive subgroups of $\oR$.

\paragraph{Distinguished bases and labeling rows/columns.}
Consider an $R$-linear map $h\colon R^m \rightarrow R^n$.  Let $\oh\colon \oR^m \rightarrow \oR^n$ be
the induced map, and let $\Phi(h)\colon R^{\mu m} \rightarrow R^{\mu n}$ and
$\oPhi(\oh)\colon \oR^{\mu m} \rightarrow \oR^{\mu n}$ be the maps obtained by applying $\Phi$ and $\oPhi$
to the entries of matrices representing $h$ and $\oh$, respectively.  As we discussed in \S \ref{section:ovicsemisimple},
the rows of $\oPhi(\oh)$ are labeled by pairs $(k,i)$ with $1 \leq k \leq q$ and $i \in \{1,\ldots,\mu_k n\}$
and the columns of $\oPhi(\oh)$ are labeled by pairs $(k,j)$ with $1 \leq k \leq q$ and $j \in \{1,\ldots,\mu_k m\}$.
We will similarly label the rows and columns of $\Phi(h)$.

For $0 \leq k \leq q$, let
\begin{equation}
\label{eqn:overlinebase}
\{\overline{\vec{v}(k)_1},\ldots,\overline{\vec{v}(k)_{\mu_k m}}\} \quad \text{and} \quad
\{\overline{\vec{w}(k)_1},\ldots,\overline{\vec{w}(k)_{\mu_k n}}\}
\end{equation}
be the distinguished bases for $\oR^{\mu m}$ and $\oR^{\mu n}$ discussed in \S \ref{section:ovicsemisimple}.
These were introduced to make sense of $\oPhi(\oh)$.  We will need the exact same bases
for $R^{\mu m}$ and $R^{\mu n}$, so let
\[\{\vec{v}(k)_1,\ldots,\vec{v}(k)_{\mu_k m}\} \quad \text{and} \quad
\{\vec{w}(k)_1,\ldots,\vec{w}(k)_{\mu_k n}\}\]
be the subsets of the standard bases for $R^{\mu m}$ and $R^{\mu n}$ that map to
\eqref{eqn:overlinebase} under the maps $R^{\mu m} \rightarrow \oR^{\mu m}$ and $R^{\mu n} \rightarrow \oR^{\mu n}$.
For all $1 \leq k \leq q$ and $1 \leq j \leq \mu_k m$, we thus have
\begin{equation}
\label{eqn:mapdistinguish}
\Phi(h)(\vec{v}(k)_j) \in \bigoplus_{k'=1}^q \left(\bigoplus_{i=1}^{\mu_{k'} n} \vec{w}(k')_i \cdot \bbL_{k'k}\right).
\end{equation}

\paragraph{S-function.}
Given a surjective map $h\colon R^m \rightarrow R^n$, the induced map $\oh\colon \oR^m \rightarrow \oR^n$
is also surjective.  For $1 \leq k \leq q$, we define 
\[\fS(h,k) = \fS(\oh,k) \subset \{1,\ldots,\mu_k m\},\]
so $|\fS(h,k)| = \mu_k n$.

\paragraph{Column-adapted maps.}
A surjective map $h\colon R^m \rightarrow R^n$ is said to be {\em column-adapted} if it satisfies
the following two conditions:
\begin{compactitem} 
\item[(i)] The map $\oh\colon \oR^m \rightarrow \oR^n$ is column-adapted in the sense of \S \ref{section:ovicsemisimple}.
\item[(ii)] For each $1 \leq k \leq q$, write $\fS(h,k) = \{j_1 < j_2 < \cdots < j_{\mu_k n}\}$.  We then require that
$\Phi(h)(\vec{v}(k)_{j_i}) = \vec{w}(k)_{i}$ for all $1 \leq i \leq \mu_k n$.
\end{compactitem}
This class of maps is closed under composition:

\begin{lemma}
\label{lemma:composecoladapted2}
Let $h_1\colon R^m \rightarrow R^n$ and $h_2\colon R^n \rightarrow R^{\ell}$ be column-adapted
maps.  Then $h_2 \circ h_1\colon R^m \rightarrow R^{\ell}$ is column-adapted.
\end{lemma}
\begin{proof}
By Lemma \ref{lemma:composecoladapted}, the map $\overline{h_2 \circ h_1} = \oh_2 \circ \oh_1$ is column-adapted,
so condition (i) is satisfied for $h_2 \circ h_1$.  The same argument used in the proof of Lemma \ref{lemma:composecoladapted}
then shows that condition (ii) is satisfied for $h_2 \circ h_1$.  The lemma follows.
\end{proof}

\paragraph{Canonical splittings.}
One of the key features of column-adapted maps is the following lemma.  We will call the map
$g$ constructed in it the {\em canonical splitting} of $h$; as the lemma says, it only
depends on the $\fS(h,k)$.

\begin{lemma}
\label{lemma:canonicalsplitting}
For each $1 \leq k \leq q$, let $S(k) \subset \{1,\ldots,\mu_k m\}$ be a $\mu_k n$-element
set.  There then exists an $R$-linear map $g\colon R^n \rightarrow R^m$ such that if
$h\colon R^m \rightarrow R^n$ is a column-adapted map with $\fS(h,k) = S(k)$ for all $1 \leq k \leq q$,
then $h \circ g = \text{id}$.
\end{lemma}
\begin{proof}
Let $\vec{v}(k)_i$ and $\vec{w}(k)_i$ be the distinguished bases for $R^{\mu m}$ and $R^{\mu n}$,
respectively.  For $1 \leq k \leq q$, write
\[S(k) = \{j(k)_1,\ldots,j(k)_{\mu_k n}\}.\]
Define $G\colon R^{\mu n} \rightarrow R^{\mu m}$ via the formula
\[G(\vec{w}(k)_i) = \vec{v}(k)_{j(k)_i} \quad \quad (1 \leq k \leq q, 1 \leq i \leq \mu_k n).\]
Since for all $1 \leq k \leq q$ and $1 \leq i \leq \mu_k n$ we trivially have
\[G(\vec{w}(k)_i) \in \bigoplus_{k'=1}^q\left(\bigoplus_{j=1}^{\mu_{k'} m} \vec{v}(k')_j \cdot \bbL_{k' k}\right),\]
it follows that there exists some $g\colon R^n \rightarrow R^m$ with $\Phi(g) = G$.  If
$h\colon R^m \rightarrow R^n$ is a column-adapted map with $\fS(h,k) = S(k)$ for all $1 \leq k \leq q$, then
for all $1 \leq k \leq q$ and $1 \leq i \leq \mu_k n$ we have
\[\Phi(h) \circ \Phi(g)(\vec{w}(k)_i) = \Phi(h)(\vec{v}_{j(k)_i}) = \vec{w}(k)_i,\]
so $\Phi(h) \circ \Phi(g) = \text{id}$ and thus $h \circ g = \text{id}$.
\end{proof}

\paragraph{Ordered VIC, Artinian case.}
By Lemma \ref{lemma:composecoladapted2}, it makes sense to define $\OVIC(R)$ to be the subcategory of $\VIC(R)$ whose
objects are all the $R^n$ with $n \geq 0$ and whose morphisms $f\colon [R^n] \rightarrow [R^m]$
are all the $\VIC(R)$-morphisms $f = (f',f'')$ such that $f''$ is column-adapted.  Since
the only column-adapted maps $R^n \rightarrow R^n$ are the identity, it follows that
the identity is the only $\OVIC(R)$-endomorphism of $[R^n]$.  

\paragraph{Factoring VIC-morphisms.}
The following proposition verifies part (b) of
Theorem \ref{theorem:orderedsummary}:

\begin{proposition}
\label{proposition:factorovic}
Let $R$ be an Artinian ring.  Every $\VIC(R)$-morphism $f\colon [R^d] \rightarrow [R^n]$ can be factored as
\[[R^d] \stackrel{f_1}{\longrightarrow} [R^d] \stackrel{f_2}{\longrightarrow} [R^n],\]
where $f_1\colon [R^d] \rightarrow [R^d]$ is a $\VIC(R)$-morphism and $f_2\colon [R^d] \rightarrow [R^n]$ is an
$\OVIC(R)$-morphism.
\end{proposition}
\begin{proof}
Write $f = (f',f'')$, where $f'\colon R^d \rightarrow R^n$ is an injection and
$f''\colon R^n \rightarrow R^d$ is a splitting of $f'$, so $f'' \circ f' = \text{id}$.

Let $\vec{v}(k)_i$ and $\vec{w}(k)_i$ be the distinguished bases of $R^{\mu n}$ and $R^{\mu d}$,
respectively.  Also, write
\[\fS(f'',k) = \{j(k)_1 < \cdots < j(k)_{\mu_k d}\} \subset \{1,\ldots,\mu_k n\}.\]
Define $G\colon R^{\mu d} \rightarrow R^{\mu d}$ via the formula
\[G(\vec{w}(k)_i) = \Phi(f'')(\vec{v}(k)_{j(k)_i}) \quad \quad (1 \leq k \leq q, 1 \leq i \leq \mu_k d).\]
Using \eqref{eqn:mapdistinguish} for
$h = f''$, for $1 \leq k \leq q$ and $1 \leq i \leq \mu_k d$ we have
\[G(\vec{w}(k)_i) \in \bigoplus_{k'=1}^q \left(\bigoplus_{j=1}^{\mu_{k'} d} \vec{w}(k')_j \cdot \bbL_{k' k}\right).\]
From this, we see that there exists some $g\colon R^d \rightarrow R^d$ such that $G = \Phi(g)$.

Since the columns of $\Phi(\og)$ are a basis for $\oR^{\mu d}$, it
follows that $\og$ is an isomorphism, so by Lemma \ref{lemma:checkinvertible} it follows that
$g$ is an isomorphism.  By construction, the map $g^{-1} \circ f''$ is column-adapted, so
$f_2 = (f' \circ g, g^{-1} \circ f'')$ is an $\OVIC(R)$-morphism.  Setting
$f_1 = (g^{-1},g)$, the map $f_1$ is a $\VIC(R)$-morphism and $f = f_2 \circ f_1$, as
desired.
\end{proof}

\paragraph{Free and dependent rows.}
Consider an $\OVIC(R)$-morphism $f\colon [R^n] \rightarrow [R^m]$ with $f = (f',f'')$.  The
condition that $f''$ is column-adapted is a condition on the columns of $\Phi(f'') \in \Mat_{\mu n, \mu m}(R)$.  We now
discuss the rows of $\Phi(f') \in \Mat_{\mu m, \mu n}(R)$.  We will call the rows of $\Phi(f')$ that
are labeled $(k,i)$ for some $1 \leq k \leq q$ and $i \in \fS(f'',k) \subset \{1,\ldots,\mu_k m\}$ the {\em dependent rows},
and all the other rows will be called the {\em free rows}.  The reason for this terminology is the following 
lemma:

\begin{lemma}
\label{lemma:comparefree}
Let $R$ be an Artinian ring.  Consider $\OVIC(R)$-morphisms $f_1,f_2\colon [R^n] \rightarrow [R^m]$ with
$f_i = (f_i',f_i'')$.  Assume that $f_1'' = f_2''$ and that the free rows of $\Phi(f_1')$ and $\Phi(f_2')$
are equal.  Then $f_1 = f_2$.
\end{lemma}
\begin{proof}
What this lemma is saying is that the dependent rows of $\Phi(f_i')$ are determined by the free rows
together with the fact that $f_i'' \circ f_i' = \text{id}$.  This is immediate from Lemma \ref{lemma:comparefreematrix}
below.
\end{proof}

\begin{lemma}
\label{lemma:comparefreematrix}
Let $R$ be any ring.  For some $a \leq b$, let $X \in \Mat_{a,b}(R)$ and $Y,Y' \in \Mat_{b,a}(R)$ be matrices
such that $X Y = \text{id}$ and $X Y' = \text{id}$.  Also, let $I \subset [b]$ be a set
with $|I| = a$ such that the submatrix of $X$ consisting of the columns of $X$ lying in $I$ is a permutation
matrix, i.e., can be transformed into the identity matrix by reordering its columns.  Assume
that the submatrices of $Y$ and $Y'$ consisting of the rows lying in $[b] \setminus I$ are equal.
Then $Y = Y'$.
\end{lemma}
\begin{proof}
This is a simple fact about matrix multiplication
that is easier to grasp from an example rather than a formal proof: if for instance we have
\[X =
\left(\begin{matrix}
\ast & 0 & 0 & \ast & 1 & \ast \\
\ast & 1 & 0 & \ast & 0 & \ast \\
\ast & 0 & 1 & \ast & 0 & \ast
\end{matrix}\right)
\quad \text{and} \quad
Y = 
\left(\begin{matrix}
\ast     & \ast     & \ast \\
\diamond & \diamond & \diamond \\
\diamond & \diamond & \diamond \\
\ast     & \ast     & \ast \\
\diamond & \diamond & \diamond \\
\ast     & \ast     & \ast
\end{matrix}\right),\]
then the $\diamond$-entries of $Y$ are determined by the $\ast$-entries of $X$ and $Y$ along
with the fact that $X Y = \text{id}$.
\end{proof}

\section{Ordered \texorpdfstring{$\VIC$}{VIC}: local Noetherianity}
\label{section:ovicnoetherian}

The goal of this section is to prove that the category of $\OVIC(R)$-modules over $\bk$ is locally
Noetherian for a finite ring $R$ and a left Noetherian ring $\bk$.  This is proved in \S \ref{section:ovicproof}, which is
preceded by two preliminary sections: \S \ref{section:wellorders} discusses well partial
orders and \S \ref{section:ordering} constructs a specific ordering that is needed for
the proof.

\subsection{Well partial orders}
\label{section:wellorders}

Let $(\fP,\preceq)$ be a poset.  We say that $\fP$ is {\em well partially ordered} if for any infinite sequence
\[p_1, p_2, p_3, \ldots \quad \quad (p_i \in \fP),\]
we can find indices 
$i_1 < i_2 < i_3 < \cdots$
such that
\begin{equation}
\label{eqn:bigsequence}
p_{i_1} \preceq p_{i_2} \preceq p_{i_3} \preceq \cdots.
\end{equation}
In fact, it is enough to just prove that 
\begin{equation}
\label{eqn:smallsequence}
\text{there exist indices $i<j$ with $p_i \preceq p_j$}.
\end{equation}
Here is a quick proof of this.  Letting
$I = \Set{$i$}{there does not exist $j>i$ with $p_j \succeq p_i$}$,
if $I$ is infinite then it provides a sequence of elements of $\fP$ violating \eqref{eqn:smallsequence},
so $I$ must be finite
and we can find the sequence \eqref{eqn:bigsequence} starting with any index larger than all the indices in $I$.

We will need the following specific well partial ordering.
Fix a finite set $\Sigma$, and let $\Sigma^{\ast}$ be the set of words $s_1 \cdots s_p$ whose letters $s_i$ are
in $\Sigma$.  Define a partial ordering on $\Sigma^{\ast}$ by saying that $s_1 \cdots s_p \preceq t_1 \cdots t_q$
if there exists a strictly increasing function $\lambda \colon \{1,\ldots,p\} \rightarrow \{1,\ldots,q\}$ with the
following two properties:
\begin{compactitem}
\item $s_i = t_{\lambda(i)}$ for $1 \leq i \leq p$, and
\item for all $1 \leq j \leq q$, there exists some $1 \leq i \leq p$ such that $\lambda(i) \leq j$ and $t_{\lambda(i)} = t_j$.
\end{compactitem}
We then have the following theorem, which is a variant on Higman's Lemma \cite{HigmanLemma}.

\begin{lemma}[{\cite[Proposition 8.2.1]{SamSnowdenGrobner}}]
\label{lemma:words}
For all finite sets $\Sigma$, the ordering $(\Sigma^{\ast},\preceq)$ is a well partial ordering.
\end{lemma}

\begin{remark}
An alternate proof of Lemma \ref{lemma:words} can be found in \cite[Proof of Prop. 7.5]{draismakuttler}.
\end{remark}

\subsection{An ordering of the generators}
\label{section:ordering}

The key to our proof that the category of $\OVIC(R)$-modules over $\bk$ is locally Noetherian for a finite
ring $R$ and a left Noetherian ring $\bk$ is the following
lemma.

\begin{lemma}
\label{lemma:constructorder}
Let $R$ be a finite ring and let $d \geq 0$.  Define
\[\fP(d) = \bigsqcup_{n=0}^{\infty} \Hom_{\OVIC(R)}(R^d,R^n).\]
There then exists a well partial ordering $\preceq$ on $\fP(d)$ along with an extension $\leq$ of $\preceq$
to a total ordering such that the following holds.  Consider $\OVIC(R)$-morphisms $f\colon [R^d] \rightarrow [R^n]$
and $g\colon [R^d] \rightarrow [R^m]$ with $f \preceq g$.  There then exists an $\OVIC(R)$-morphism 
$\phi\colon [R^n] \rightarrow [R^m]$ with the following two properties:
\begin{compactitem}
\item[(i)] $g = \phi \circ f$, and
\item[(ii)] if $h\colon [R^d] \rightarrow [R^{n}]$ is an $\OVIC(R)$-morphism such that $h<f$, then 
\[\phi \circ h < \phi \circ f = g.\]
\end{compactitem}
\end{lemma}

\begin{remark}
\label{remark:extend}
A total ordering $\leq$ that extends a well partial ordering $\preceq$ (as in Lemma \ref{lemma:constructorder}) is
a well-ordering.
\end{remark}

\begin{proof}[Proof of Lemma \ref{lemma:constructorder}]
The notation will be as in \S \ref{section:ovicgeneral}.  Our finite ring $R$ is Artinian,
so $\oR = R/J(R)$ is semisimple and
\[\oR \cong \Mat_{\mu_1}(\bbD_1) \times \cdots \times \Mat_{\mu_q}(\bbD_q)\]
for $\mu_1,\ldots,\mu_q \geq 1$ and division rings $\bbD_1,\ldots,\bbD_q$.
Set $\mu = \mu_1+\cdots+\mu_q$.  Let $\Phi\colon R \hookrightarrow \Mat_{\mu}(R)$ and
$\oPhi\colon \oR \hookrightarrow \Mat_{\mu}(\oR)$ be the Artin--Wedderburn
embeddings of $R$ and $\oR$, so $\overline{\Phi(x)} = \oPhi(\ox)$ for all $x \in R$.

\begin{step}{1}
\label{step:total}
We construct the total order $\leq$ on $\fP(d)$.
\end{step}

Fix an arbitrary total order on $R^{\mu d}$.  Consider $\OVIC(R)$-morphisms $f\colon [R^d] \rightarrow [R^n]$
and $g\colon [R^d] \rightarrow [R^m]$ in $\fP(d)$.  Write $f = (f',f'')$ and $g = (g',g'')$.  We
determine if $f < g$ via the following procedure:
\begin{compactitem}
\item If $n<m$, then $f < g$.
\item Otherwise, assume that $n=m$.  For each $1 \leq k \leq q$, we have the
$\mu_k d$-element subsets $\fS(f'',k)$ and $\fS(g'',k)$ of $\{1,\ldots,\mu_k m\}$.  Order
subsets of $\{1,\ldots,\mu_k m\}$ lexicographically, and
then further order tuples $(I_1,\ldots,I_q)$ with $I_k \subset \{1,\ldots,\mu_k m\}$ lexicographically.
If
\[(\fS(f'',1),\ldots,\fS(f'',q)) < (\fS(g'',1),\ldots,\fS(g'',q))\]
using this order, then $f<g$.
\item Otherwise, assume that $n=m$ and that $\fS(f'',k) = \fS(g'',k)$ for all $1 \leq k \leq q$.  
Compare the columns of matrices in $\Mat_{\mu d, \mu n}(R)$ using 
our fixed total order on $R^{\mu d}$ and the lexicographic order.  If under
this ordering the columns of $\Phi(f'')$ are less than the columns of $\Phi(g'')$, then $f<g$.
\item Otherwise, assume that $n=m$ and that $f''=g''$.  Compare the free rows
of $\Phi(f') \in \Mat_{\mu n, \mu d}(R)$ and $\Phi(g') \in \Mat_{\mu n, \mu d}(R)$ using
our fixed total order on $R^{\mu d}$ and the lexicographic order.  If under this ordering
the free rows of $\Phi(f')$ are less than the rows of $\Phi(g')$, then $f<g$.
\end{compactitem}
By Lemma \ref{lemma:comparefree}, this determines a total order $\leq$ on $\fP(d)$.

\begin{step}{2}
\label{step:stabilize}
We define the notion of a stabilization of an $\OVIC(R)$-morphism.
\end{step}

Let $f\colon [R^d] \rightarrow [R^n]$ be an $\OVIC(R)$-morphism and let $1 \leq a \leq b \leq n$.  An
{\em $(a,b)$-stabilization} of $f$ (or simply a {\em stabilization} if we do not want to specify $a$ and
$b$) is an $\OVIC(R)$-morphism $g\colon [R^d] \rightarrow [R^{n+1}]$
related to $f$ as follows.  Write $f = (f',f'')$ and $g = (g',g'')$.  Regard
$f'\colon R^d \rightarrow R^n$ and $f'' \colon R^n \rightarrow R^d$ and $g'\colon R^d \rightarrow R^{n+1}$
and $g'' \colon R^{n+1} \rightarrow R^d$ as matrices.  We then require the following three
conditions to hold:
\begin{compactitem}
\item $g''$ is obtained from $f''$ by inserting a copy of the $a^{\text{th}}$ column of $f''$ after the
$b^{\text{th}}$ column.
\item Let $\hg'\colon R^d \rightarrow R^{n+1}$ be the matrix obtained from $f'$ by inserting a copy of
the $a^{\text{th}}$ row of $f'$ after the $b^{\text{th}}$ row.  We then require
all the free rows\footnote{We cannot require $g' = \hg'$ since we need $g'' \circ g' = \text{id}$, which 
requires changing the dependent rows.} of $\Phi(g')$ to equal the corresponding rows of $\Phi(\hg')$.
\item None of the dependent rows of $\Phi(f')$ are contained in rows coming from the $a^{\text{th}}$
row of $f'$.
\end{compactitem}
Note that these three conditions imply that $\fS(f'',k) = \fS(g'',k)$ for all $1 \leq k \leq q$. 
We remark that Lemma \ref{lemma:comparefree} implies that an $(a,b)$-stabilization is unique if
it exists.  Technically, our proof does not require proving that it exists, but we remark that
its existence can be derived from the argument of Step \ref{step:phi} below, which constructs
an $\OVIC(R)$-morphism $\phi\colon [R^n] \rightarrow [R^{n+1}]$ such that $g = \phi \circ f$
is the $(a,b)$-stabilization of $f$.

\begin{step}{3}
\label{step:partial}
We construct the partial order $\preceq$ such that $\leq$ is a refinement of $\preceq$.
\end{step}

Consider $\OVIC(R)$-morphisms $f\colon [R^d] \rightarrow [R^n]$
and $g\colon [R^d] \rightarrow [R^m]$ in $\fP(d)$.   We then
say that $f \prec g$ if for some $r \geq 1$ there exists a sequence of $\OVIC(R)$-morphisms
\[f = h_0,h_1,\ldots,h_r = g,\]
where for $0 \leq i < r$ the $\OVIC(R)$-morphism $h_{i+1}$ is a stabilization of
the $\OVIC(R)$-morphism $h_i$.  Since a stabilization increases the rank of
the codomain by $1$, this implies that $m = n + r$, and in particular
that $n < m$.  This clearly defines a partial ordering $\preceq$ on $\fP(d)$, and since $f \prec g$ required $n<m$
our total ordering $\leq$ from Step \ref{step:total} refines $\preceq$.

\begin{step}{4}
\label{step:well}
We prove that $\preceq$ is a well partial order.
\end{step}

We will embed $(\fP(d),\preceq)$ into a poset $(\Sigma^{\ast},\preceq)$ of words, where $\Sigma$ is
a finite set of letters and $\preceq$ is as\footnote{Recall that for
words $s_1 \cdots s_p$ and $t_1 \cdots t_q$ in $\Sigma^{\ast}$, we have
$s_1 \cdots s_p \preceq t_1 \cdots t_q$ if there exists a 
strictly increasing function $\lambda \colon \{1,\ldots,p\} \rightarrow \{1,\ldots,q\}$ with the
following two properties:
\begin{compactitem}
\item $s_i = t_{\lambda(i)}$ for $1 \leq i \leq p$, and
\item for all $1 \leq j \leq q$, there exists some $1 \leq i \leq p$ such that $\lambda(i) \leq j$ and $t_{\lambda(i)} = t_j$.
\end{compactitem}}
in Lemma \ref{lemma:words}.  That lemma says
that $(\Sigma^{\ast},\preceq)$ is a well partial ordering, so this will imply that $(\fP(d),\preceq)$ is
as well.

First, define
\[\hR = R \sqcup \{\clubsuit\},\]
where $\clubsuit$ is a formal symbol.  Though $\hR$ is not a ring, it still makes
sense to speak about matrices with entries in $\hR$.  Define
\[\Sigma = \Set{$(M_1,M_2)$}{$M_1 \in \Mat_{\mu,\mu d}(\hR)$ and $M_2 \in \Mat_{d,1}(R)$}.\]
We then define a map $\iota\colon \fP(d) \rightarrow \Sigma^{\ast}$ in the following way.  

Consider an element $f\colon [R^d] \rightarrow [R^n]$ of $\fP(d)$.  Write $f = (f',f'')$. 
Let $r_1,\ldots,r_n \in \Mod_{1,d}(R)$ be the rows of the matrix representing
$f'\colon R^d \rightarrow R^n$ and let $c_1,\ldots,c_n \in \Mod_{d,1}(R)$ be the columns
of the matrix representing $f''\colon R^n \rightarrow R^d$.  We have
\[\Phi(r_1),\ldots,\Phi(r_n) \in \Mod_{\mu,\mu d}(R),\]
and the word
\[(\Phi(r_1),c_1) \cdots (\Phi(R_n),c_n) \in \Sigma^{\ast}\]
contains all the information about $f$.  However, this encoding is redundant since
it contains not just the data of the free rows of $f'$, but also the data of
the dependent rows.

For each $1 \leq i \leq n$, we modify $\Phi(r_i) \in \Mat_{\mu,\mu d}(R)$ to
form $\hr_i \in \Mat_{\mu,\mu d}(\hR)$ in the following way.
Each row of $\Phi(r_i)$ is a row of $\Phi(f')$.  If that row is a free row, then do not
change it.  Otherwise, if it is a dependent row, then replace each entry in it
with the formal symbol $\clubsuit$.  The result is a matrix
$\hr_i \in \Mat_{\mu,\mu d}(\hR)$ some of whose entire rows consist of repeated $\clubsuit$'s.

We now define
\[\iota(f) = (\hr_1,c_1)(\hr_2,c_2)\cdots(\hr_n,c_n) \in \Sigma^{\ast}.\]
This is an injection since knowing $\iota(f)$, we can reconstruct $f''$ and all the free rows
of $\Phi(f')$, and this determines $f'$ and hence $f = (f',f'')$ by Lemma \ref{lemma:comparefree}.  
That $\iota$ is order-preserving is immediate from the definitions.

\begin{step}{5}
\label{step:phi}
We construct the $\phi$ satisfying (i).
\end{step}

Consider $\OVIC(R)$-morphisms $f\colon [R^d] \rightarrow [R^n]$
and $g\colon [R^d] \rightarrow [R^m]$ with $f \preceq g$.  Our goal is to construct
an $\OVIC(R)$-morphism $\phi\colon [R^n] \rightarrow [R^m]$ such that
$g = \phi \circ f$.  Examining the definition of the partial ordering $\preceq$ in
Step \ref{step:partial}, we see that it is enough to deal with the case where
$g$ is an $(a,b)$-stabilization of $f$ for some $1 \leq a \leq b \leq n$.  The
general case can be dealt with by iterating this $m-n$ times.

Write $f = (f',f'')$ and $g = (g',g'')$.  By definition, the following three things hold:
\begin{compactitem}
\item $g''$ is obtained from $f''$ by inserting a copy of the $a^{\text{th}}$ column of $f''$ after the
$b^{\text{th}}$ column.
\item Let $\hg'\colon R^d \rightarrow R^{n+1}$ be the matrix obtained from $f'$ by inserting a copy of
the $a^{\text{th}}$ row of $f'$ after the $b^{\text{th}}$ row.  Then 
all the free rows of $\Phi(g')$ equal the corresponding rows of $\Phi(\hg')$.
\item None of the dependent rows of $\Phi(f')$ are contained in rows coming from the $a^{\text{th}}$
row of $f'$.
\end{compactitem}
Let $\psi\colon R^d \rightarrow R^n$ be the canonical splitting of $f''$ (see Lemma \ref{lemma:canonicalsplitting}).
Let $\fc \in R^d$ be the $a^{\text{th}}$ column of the matrix representing $f''$, and set $\hfc = \psi(\fc) \in R^n$.
We then define $\phi = (\phi',\phi'')$ in the following way:
\begin{compactitem}
\item $\phi''\colon R^{n+1} \rightarrow R^n$ is represented by the matrix obtained by inserting $\hfc$ after
the $b^{\text{th}}$ column of $\id\colon R^{n} \rightarrow R^n$.
\item $\phi'\colon R^n \rightarrow R^{n+1}$ is represented by the matrix obtained by first subtracting
$\hfc$ from the $a^{\text{th}}$ column of $\id\colon R^n \rightarrow R^n$, and then inserting the
row $(0,\ldots,0,1,0,\ldots,0)$ with a $1$ in position $a$ after the $b^{\text{th}}$ row.
\end{compactitem}
For example, for $n=7$ and $a=3$ and $b=4$ we would have
\[
  \phi'' = \left(\begin{array}{@{}cccc|c|ccc@{}}
1 & 0 & 0 & 0 & \hfc_1 & 0 & 0 & 0 \\
0 & 1 & 0 & 0 & \hfc_2 & 0 & 0 & 0 \\
0 & 0 & 1 & 0 & \hfc_3 & 0 & 0 & 0 \\
0 & 0 & 0 & 1 & \hfc_4 & 0 & 0 & 0 \\
0 & 0 & 0 & 0 & \hfc_5 & 1 & 0 & 0 \\ 
0 & 0 & 0 & 0 & \hfc_6 & 0 & 1 & 0 \\
0 & 0 & 0 & 0 & \hfc_7 & 0 & 0 & 1 \\
\end{array}\right)
\quad
\phi' = \left(\begin{array}{@{}ccccccc@{}}
1 & 0 & -\hfc_1  & 0 & 0 & 0 & 0 \\
0 & 1 & -\hfc_2  & 0 & 0 & 0 & 0 \\
0 & 0 & 1-\hfc_3 & 0 & 0 & 0 & 0 \\
0 & 0 & -\hfc_4  & 1 & 0 & 0 & 0 \\
\hline
0 & 0 & 1       & 0 & 0 & 0 & 0 \\
\hline
0 & 0 & -\hfc_5  & 0 & 1 & 0 & 0 \\
0 & 0 & -\hfc_6  & 0 & 0 & 1 & 0 \\
0 & 0 & -\hfc_7  & 0 & 0 & 0 & 1 \\
              \end{array}\right).
            \]
It is clear that $\phi'' \circ \phi' = \text{id}$ and that the matrix representing 
$f'' \circ \phi''\colon R^{n+1} \rightarrow R^d$ is obtained by inserting $f''(\hfc) = \fc$
after the $b^{\text{th}}$ column of the matrix representing $f''$.  Moreover, examining
the construction of the canonical splitting in Lemma \ref{lemma:canonicalsplitting}, we see
that the entries of $\Phi(\hfc) \in R^{\mu n}$ lying in the free rows of $\Phi(f')$ are all
$0$, so the matrix corresponding to $\phi' \circ f'$ is obtained by first inserting a copy of the
$a^{\text{th}}$ row of the matrix representing $f'$ after the $b^{\text{th}}$ row of that matrix,
and then possibly modifying the dependent rows.

\begin{step}{6}
We prove that the $\phi$ we constructed satisfy (ii).
\end{step}

Just like in the previous step, it is enough to deal with the case $g\colon [R^d] \rightarrow [R^{n+1}]$
is a stabilization of $f\colon [R^d] \rightarrow [R^n]$.
Consider some $\OVIC(R)$-morphism $h\colon [R^d] \rightarrow [R^n]$ such that $h<f$.  Our goal is
to prove that $\phi \circ h < \phi \circ f$.  Write $f = (f',f'')$ and $h = (h',h'')$ and
$\phi = (\phi',\phi'')$.

Examining the construction of the total ordering $\leq$ in Step \ref{step:total}, we see
that there are three cases we have to deal with.  The first is where
\[(\fS(h'',1),\ldots,\fS(h'',q)) < (\fS(f'',1),\ldots,\fS(f'',q)),\]
where subsets of $\{1,\ldots,\mu_k n\}$ are ordered
using the lexicographic ordering and these tuples are further ordered
using the lexicographic ordering.  Lemma \ref{lemma:ordersfunction}
then implies that
\[(\fS(h'' \circ \phi'',1),\ldots,\fS(h'' \circ \phi'',q)) < (\fS(f'' \circ \phi'',1),\ldots,\fS(f'' \circ \phi'',q)),\]
so $\phi \circ h < \phi \circ f$.

The second case is where $\fS(h'',k) = \fS(f'',k)$ for all $k$, but 
the columns of $\Phi(h'')$ are less than the columns of $\Phi(f'')$ in the lexicographic
ordering (using our fixed total ordering on $R^{\mu d}$).  In this case, it
follows from our construction of $\phi$ that the matrix representing
$h'' \circ \phi''$ is obtained from the matrix representing $h''$ by inserting a copy
of the $a^{\text{th}}$ column of the matrix representing $f''$ after the $b^{\text{th}}$
column, and similarly for $f'' \circ \phi''$.  
This implies that the columns of $\Phi(h'' \circ \phi'')$ remain less than the columns of $\Phi(f'' \circ \phi'')$,
so $\phi \circ h < \phi \circ f$.

The final case is where $h'' = f''$, but the free rows of $\Phi(h')$ are less than
the free rows of $\Phi(f')$ in the lexicographic ordering.  In this case, $\Phi(\phi' \circ h')$
is obtained from $\Phi(h')$ by taking a bunch of free rows and duplicating them lower in
the matrix, and similarly for $\Phi(\phi' \circ f')$ (with the same rows).  It follows
that the free rows of $\Phi(\phi' \circ h')$ remain less than the free rows
of $\Phi(\phi' \circ f')$, so $\phi \circ h < \phi \circ f$.
\end{proof}

\subsection{Local Noetherianity}
\label{section:ovicproof}

We now prove the following, which verifies part (c) of Theorem \ref{theorem:orderedsummary}: 

\begin{proposition}
\label{proposition:ovicnoetherian}
Let $R$ be a finite ring and let $\bk$ be a left Noetherian ring.  Then the category of $\OVIC(R)$-modules over $\bk$
is locally Noetherian.
\end{proposition}
\begin{proof}
Just like in the proof of Theorem \ref{maintheorem:noetherian} in
\S \ref{section:reductionordered}, we will prove this by studying representable modules.  For $d \geq 0$, let
$P(d)$ be the $\OVIC(R)$-module defined via the formula
\[P(d)_n = \bk[\Hom_{\OVIC(R)}(R^d,R^n)] \quad \quad (n \geq 0).\]
As we discussed in the proof of Theorem \ref{maintheorem:noetherian}, every finitely generated
$\OVIC(R)$-module over $\bk$ is the surjective image of a direct sum of finitely many $P(d)$ (for differing choices of $d$).
To prove that every submodule of such a finitely generated module is finitely generated, it is thus
enough to prove this for $P(d)$.

We start with some preliminaries.  Let $\preceq$ and $\leq$ be the orderings on 
\[\fP(d) = \bigsqcup_{n=0}^{\infty} \Hom_{\OVIC(R)}(R^d,R^n)\]
provided by
Lemma \ref{lemma:constructorder}.  For a nonzero $x \in P(d)_n$, define the {\em initial term} of
$x$, denoted $\init(x)$, as follows.  Write
\begin{align*}
&x = \alpha_1 f_1 + \cdots + \alpha_k f_k \quad \text{with $\alpha_1,\ldots,\alpha_k \in \bk \setminus \{0\}$}\\
&\quad\quad\quad\text{and $f_1,\ldots,f_k \in \Hom_{\OVIC(R)}(R^d,R^n)$ pairwise distinct}.
\end{align*}
Order these terms such that $f_1 < f_2 < \cdots < f_k$.  Then $\init(x) = \alpha_k f_k$.

Next, for an $\OVIC(R)$-submodule $M$ of $P(d)$, define the {\em initial module} $\sinit(M)_{\bullet}$ of $M$ to be the
ordered sequence of $\bk$-modules defined via the formula
\[\sinit(M)_n = \bk\Set{$\init(x)$}{$x \in M_n$} \quad \quad (n \geq 1).\]  
Be warned that this need not be an $\OVIC(R)$-submodule of $P(d)$.  However, we do have the following.

\begin{claim}
If $N$ and $M$ are $\OVIC(R)$-submodules of $P(d)$ with $N \subset M$ and $\sinit(N)_{\bullet} = \sinit(M)_{\bullet}$, 
then $N = M$.
\end{claim}
\begin{proof}[Proof of claim]
Assume otherwise, and let $n \geq 0$ be such that $N_n \subsetneq M_n$.  Recall that $\leq$ is a well-order (c.f. Remark \ref{remark:extend}).
Consider the nonempty set
\[\Set{$f$}{there exists $x \in M_n \setminus N_n$ and $\alpha \in \bk \setminus \{0\}$ such that $\init(x) = \alpha f$}.\]
Since $\leq$ is a well total ordering, this set has a $\leq$-minimal element $f\colon [R^d] \rightarrow [R^n]$; indeed, if it did not we could find an infinite strictly
decreasing sequence of elements in it.
Let $x \in M_n \setminus N_n$ satisfy $\init(x) = \alpha f$ with $\alpha \in \bk \setminus \{0\}$.  By assumption, there exists
some $y \in N_n$ such that $\init(y) = \alpha f$.  The $\alpha f$ terms cancel in $x-y$, so
$\init(x-y) = \beta g$ with $\beta \in \bk \setminus \{0\}$ and $g < f$.  Since $x \in M_n \setminus N_n$ and $y \in N_n$, we have
$x-y \in M_n \setminus N_n$, so this contradicts the minimality of $f$.
\end{proof}

We now commence with the proof that every $\OVIC(R)$-submodule of $P(d)$ is finitely generated.  Assume otherwise, so
there exists a strictly increasing chain
\[M^0 \subsetneq M^1 \subsetneq M^2 \subsetneq \cdots\]
of $\OVIC(R)$-submodules of $P(d)$.  By the above claim, the initial modules $\sinit(M^i)_{\bullet}$ must all be distinct, so for
all $i \geq 1$ we can find some $n_i \geq 0$ such that there exists some
\[\alpha_i f_i \in \sinit(M^i)_{n_i} \setminus \sinit(M^{i-1})_{n_i} \quad \text{with $\alpha_i \in \bk \setminus \{0\}$ and $f_i\colon [R^d] \rightarrow [R^{n_i}]$}.\]
Let $x_i \in M^i_{n_i}$ be an element with $\init(x_i) = \alpha_i f_i$.

Since $\preceq$ is a well partial ordering, there exists some increasing sequence
$i_1 < i_2 < i_3 < \cdots$
of indices such that
\[f_{i_1} \preceq f_{i_2} \preceq f_{i_3} \preceq \cdots.\]
Since $\bk$ is a left Noetherian ring, there exists some $m \geq 1$ such that $\alpha_{m+1}$ is in the
left $\bk$-ideal generated by $\alpha_{i_1},\ldots,\alpha_{i_m}$, i.e., we can write
\[\alpha_{m+1} = c_1 \alpha_{i_1} + \cdots + c_m \alpha_{i_m} \quad \text{with $c_1,\ldots,c_m \in \bk$}.\]
For $1 \leq j \leq m$, the fact that $f_{i_j} \preceq f_{i_{m+1}}$ implies by part (i) of
Lemma \ref{lemma:constructorder} that there exists some $\OVIC(R)$-morphism 
$\phi_j\colon [R^{n_{i_j}}] \rightarrow [R^{n_{i_{m+1}}}]$ such that $f_{i_{m+1}} = \phi_j \circ f_{i_j}$.
Conclusion (ii) of Lemma \ref{lemma:constructorder} implies that $\init(\phi_j \circ x) = \alpha_j f_{i_{m+1}}$.
Setting
\[y = \sum_{j=1}^m c_j (\phi_j \circ x_{i_j}) \in M^{i_m}_{n_{i_{m+1}}},\]
we thus see that
\[\init(y) = \sum_{j=1}^m c_j \alpha_j f_{i_{m+1}} = \alpha_{m+1} f_{i_{m+1}} = \init(x_{i_{m+1}}).\]
This contradicts the fact that
\[\init(x_{i_{m+1}}) \in \sinit(M^{i_{m+1}})_{n_{i_{m+1}}} \setminus \sinit(M^{i_m})_{n_{i_{m+1}}}.\]
The proposition follows.
\end{proof}

\begin{footnotesize}
\noindent
\begin{tabular*}{\linewidth}[t]{@{}p{\widthof{Department of Mathematics}+0.5in}@{}p{\linewidth - \widthof{Department of Mathematics} - 0.5in}@{}}
{\raggedright
Andrew Putman\\
Department of Mathematics\\
University of Notre Dame\\
279 Hurley Hall\\
South Bend, IN 46556\\
{\tt andyp@nd.edu}}
&
{\raggedright
Steven V Sam\\
Department of Mathematics\\
University of California, San Diego\\ 
9500 Gilman Drive\\
La Jolla, CA 92093-0112 \\
{\tt ssam@ucsd.edu}}
\end{tabular*}\hfill
\end{footnotesize}


\begin{thebibliography}{00}
\begin{footnotesize}
\setlength{\itemsep}{-1mm}

\bibitem{ChurchEllenbergFarbFI}
T. Church, J. S. Ellenberg\ and\ B. Farb, FI-modules and stability for representations of symmetric groups, Duke Math. J. 164 (2015), no.~9, 1833--1910, \arxiv{1204.4533v4}.

\bibitem{CEFN}
T. Church, J. S. Ellenberg, B. Farb and R. Nagpal, FI-modules over Noetherian rings, Geom. Topol. 18 (2014), no.~5, 2951--2984, \arxiv{1210.1854v2}.

\bibitem{ChurchFarbRepStability}
T. Church\ and\ B. Farb, Representation theory and homological stability, Adv. Math. 245 (2013), 250--314, \arxiv{1008.1368v3}.

\bibitem{draismakuttler} J. Draisma and J. Kuttler, Bounded-rank tensors are defined in bounded degree, Duke Math. J. 163 (2014), no.~1, 35--63, \arxiv{1103.5336v2}.

\bibitem{Dwyer}
W. G. Dwyer, Twisted homological stability for general linear groups, Ann. of Math. (2) 111 (1980), no.~2, 239--251. 

\bibitem{FarbICM}
B. Farb, Representation stability,
Proceedings of the 2014 International Congress of Mathematicians.
Volume II, 1173--1196, \arxiv{1404.4065v1}.

\bibitem{FraFriPiSch}
V. Franjou\ and\ A. Touz\'{e} (eds.), {\it Lectures on functor homology}, Progress in Mathematics, 311, Birkh\"auser/Springer, Cham, 2015.

\bibitem{GanLiEI}
W. L. Gan\ and\ L. Li, Noetherian property of infinite EI categories, New York J. Math. {\bf 21} (2015), 369--382,
\arxiv{1407.8235v3}.

\bibitem{Harman}
N. Harman, Effective and Infinite-Rank Superrigidity in the Context of Representation Stability, preprint 2019,
\arxiv{1902.05603v1}.

\bibitem{HigmanLemma}
G. Higman, Ordering by divisibility in abstract algebras, Proc. London Math. Soc. (3) 2 (1952), 326--336.

\bibitem{Kuhn}
N. J. Kuhn, Generic representations of the finite general linear groups and the Steenrod algebra. II, $K$-Theory 8 (1994), no.~4, 395--428.

\bibitem{LamBook}
T. Y. Lam, {\it A first course in noncommutative rings}, second edition, Graduate Texts in Mathematics, 131, Springer-Verlag, New York, 2001.

\bibitem{LamSerre}
T. Y. Lam, {\it Serre's problem on projective modules}, Springer Monographs in Mathematics, Springer-Verlag, Berlin, 2006. 

\bibitem{PutmanTwistedStability}
A. Putman, A new approach to twisted homological stability, with applications to congruence subgroups, preprint 2021, \arxiv{2109.14015}

\bibitem{PutmanSamLinear}
A. Putman\ and\ S. V. Sam, Representation stability and finite linear groups, Duke Math. J. {\bf 166} (2017), no.~13, 2521--2598, \arxiv{1408.3694v3}.

\bibitem{PutmanSamSnowdenUni}
A. Putman, S. V. Sam\ and\ A. Snowden, Stability in the homology of unipotent groups, Algebra Number Theory {\bf 14} (2020), no.~1, 119--154, \arxiv{1711.11080v4}.

\bibitem{RandalWilliamsWahl}
O. Randal-Williams\ and\ N. Wahl, Homological stability for automorphism groups, Adv. Math. {\bf 318} (2017), 534--626, \arxiv{1409.3541}.

\bibitem{Richter}
G. Richter, Noetherian semigroup rings with several objects, in {\it Group and semigroup rings (Johannesburg, 1985)}, 231--246, North-Holland Math. Stud., 126, Notas Mat., 111, North-Holland, Amsterdam.

\bibitem{SamSnowdenGrobner} 
S. V Sam and A. Snowden, Gr\"obner methods for representations of combinatorial categories, 
J. Amer. Math. Soc. 30 (2017), 159--203, \arxiv{1409.1670v3}.

\bibitem{ScorichnkoThesis}
A. Scorichenko, Stable K-theory and functor homology over a ring, thesis 2000.

\bibitem{vanDerKallen}
W. van der Kallen, Homology stability for linear groups, Invent. Math. 60 (1980), no.~3, 269--295.

\bibitem{WilsonFIW}
J. C. H. Wilson, $\rm FI_W$-modules and stability criteria for representations of classical Weyl groups, J. Algebra {\bf 420} (2014), 269--332, \arxiv{1309.3817v2}.

\end{footnotesize}
\end{thebibliography}
\end{document}